\documentclass[11pt]{amsart}
\usepackage[colorlinks=true, pdfstartview=FitV, linkcolor=blue, 
citecolor=blue]{hyperref}

\usepackage{amsmath,amssymb,bbm,amscd}
\usepackage{graphicx}
\usepackage{a4wide}

\theoremstyle{plain}
\newtheorem{theorem}{Theorem}[section]
\newtheorem{prop}[theorem]{Proposition}
\newtheorem{lemma}[theorem]{Lemma}
\newtheorem{coro}[theorem]{Corollary}
\newtheorem{fact}[theorem]{Fact}

\theoremstyle{definition}

\newtheorem{remark}[theorem]{Remark}
\newtheorem{definition}[theorem]{Definition}

\numberwithin{equation}{section}

\newcommand{\ii}{\ts\mathrm{i}}

\newcommand{\ts}{\hspace{0.5pt}}
\newcommand{\nts}{\hspace{-0.5pt}}

\DeclareMathOperator{\dens}{\mathrm{dens}}
\DeclareMathOperator{\diam}{\mathrm{diam}}
\DeclareMathOperator{\card}{\mathrm{card}}
\DeclareMathOperator{\cent}{\mathrm{cent}}
\DeclareMathOperator{\Gal}{\mathrm{Gal}}
\DeclareMathOperator{\stab}{\mathrm{stab}}
\DeclareMathOperator{\norm}{\mathrm{norm}}

\DeclareMathOperator{\vol}{\mathrm{vol}}
\DeclareMathOperator{\GL}{\mathrm{GL}}
\DeclareMathOperator{\No}{\mathrm{N}}

\DeclareMathOperator{\id}{\mathrm{id}}

\newcommand{\vG}{\varGamma}

\newcommand{\fb}{\mathfrak{b}}
\newcommand{\fp}{\mathfrak{p}}
\newcommand{\fq}{\mathfrak{q}}
\newcommand{\fP}{\mathfrak{P}}
\newcommand{\fL}{\mathfrak{L}}

\newcommand{\cB}{\mathcal{B}}

\newcommand{\cG}{\mathcal{G}}
\newcommand{\cH}{\mathcal{H}}
\newcommand{\cL}{\mathcal{L}}

\newcommand{\cR}{\mathcal{R}}
\newcommand{\cS}{\mathcal{S}}

\newcommand{\cO}{\mathcal{O}}

\newcommand{\AAA}{\mathbb{A}}

\newcommand{\ZZ}{\mathbb{Z}\ts}
\newcommand{\RR}{\mathbb{R}\ts}
\newcommand{\CC}{\mathbb{C}\ts}
\newcommand{\NN}{\mathbb{N}}
\newcommand{\QQ}{\mathbb{Q}}

\newcommand{\XX}{\mathbb{X}}
\newcommand{\YY}{\mathbb{Y}}

\newcommand{\one}{\mathbbm{1}}
\newcommand{\Aut}{\mathrm{Aut}}

\newcommand{\defeq}{\mathrel{\mathop:}=}
\newcommand{\eqdef}{=\mathrel{\mathop:}}

\newcommand{\exend}{\hfill $\Diamond$}

\newcommand{\myfrac}[2]{\frac{\raisebox{-2pt}{$#1$}}
      {\raisebox{0.5pt}{$#2$}}}
\newcommand{\bs}[1]{\boldsymbol{#1}}

\begin{document}

\title[Power-free points in quadratic number fields]{Power-free
  points in quadratic number fields:\\[2mm]
  Stabiliser, dynamics and entropy}

\author{Michael Baake}

\address{Fakult\"{a}t f\"{u}r Mathematik,
  Universit\"{a}t Bielefeld,\newline \hspace*{\parindent}Postfach
  100131, 33501 Bielefeld, Germany}
\email{mbaake@math.uni-bielefeld.de}

\author{\'{A}lvaro Bustos}

\address{School of Mathematics and Statistics, The Open
  University,\newline \indent Walton Hall, Milton Keynes MK7 6AA, UK}
\email{abustos@math.uni-bielefeld.de}

\author{Andreas Nickel}


\address{Institut f\"{u}r Theoretische Informatik, Mathematik
  und Operations Research, \newline
  \indent Universit\"{a}t der Bundeswehr M\"{u}nchen, \newline
  \indent Werner-Heisenberg-Weg 39, 85579 Neubiberg, Germany}
\email{andreas.nickel@unibw.de}

\makeatletter
\@namedef{subjclassname@2020}{%
  \textup{2020} Mathematics Subject Classification}
\makeatother

\keywords{Number-theoretic shift spaces, Extended symmetries,
  Topological entropy}

\subjclass[2020]{37B10, 11R11, 11R04}

\begin{abstract}
  The sets of $k$-free integers in general quadratic number fields are
  studied, with special emphasis on (extended) symmetries and their
  impact on the topological dynamical systems induced by such
  integers. We establish correspondences between number-theoretic and
  dynamical quantities, and use symmetries and entropy to distinguish
  the systems.
\end{abstract}

\maketitle
\thispagestyle{empty}

\section{Introduction}\label{sec:intro}

Topological and measure-theoretic dynamical systems have been studied
for a long time, both in one and in higher dimensions. Powerful
connections between such systems and number theory are known since the
pioneering work of Furstenberg; see \cite{Furst} for a concise
introduction, and \cite{Klaus} for an account of some of the
complications that show up in the innocently looking step from one to
more than one dimension.

Here, we revisit one particular aspect of this connection, namely the
structure of certain two-dimensional shift spaces of number-theoretic
origin.  More precisely, motivated by the properties of square-free
integers and visible lattice points \cite{Apo,BMP,Sarnak}, we are
interested in the planar shifts of $k$-free integers in arbitrary
quadratic number fields, thus putting some of the observations from
\cite{BBHLN} into a more general setting.  Here, given an arbitrary
quadratic field $K$ with ring of integers $\cO^{}_{\! K}$, an element
$x\in\cO^{}_{\! K}$ is called $k$-free for some fixed natural number
$k\geqslant 2$ when the principal ideal generated by $x$ is not
divisible by the $k$-th power of any prime ideal in $\cO^{}_{\! K}$.

The set of $k$-free integers gives rise to a natural topological
dynamical system via its Minkowski embedding into $\RR^2$ and the
topological closure of the $\cO^{}_{\! K}$-orbit of the resulting
discrete point set in the standard local topology. These shifts have
interesting properties that are known from the set of visible lattice
points \cite{BMP} and various generalisations to $\cB$-free lattice
systems \cite{BH,CV,BBHLN}, where the latter are also generalisations
of the recently much-studied $\cB$-free integers \cite{Abda,DKKL}. In
fact, via the Minkowski embedding, previously studied extensions to
number fields \cite{CV,BBHLN} can also be viewed as $\cB$-free lattice
systems.

It is an interesting general observation that such systems can also be
described in the setting of weak model sets \cite{TAO,BHS,Keller,KKL},
which builds on the pioneering work of Meyer \cite{Meyer} and gives
rapid access to various spectral and dynamical properties of such
systems \cite{BHS,HR}. Among these results is the statement that the
dynamical spectrum is pure point, though no eigenfunction except the
trivial one is continuous, and also a general formula for the spectrum
and for the topological entropy of such shift spaces.

Clearly, one natural goal is the investigation of these shifts up to
topological conjugacy, where some fairly simple groups come in handy,
namely the topological centraliser and normaliser of the translation
group in the group of homeomorphisms; see \cite{BRY,Baa,Bustos} and
references therein. Generalising a result of Mentzen \cite{Mentzen}
for the square-free integers on the line, it was shown in previous
work \cite{BBHLN} that the centraliser is trivial for many of these
$k$-free shifts, while the normaliser is not. Therefore, certain
results can already be obtained from this relatively simple invariant,
which has the advantage of being explicitly computable.

In all our cases, we deal with examples of single orbit dynamics
\cite{Weiss}, which implies that we can derive many properties from
the defining point set $V$ of $k$-free integers (in its Minkowski
embedding). Our strategy thus is to first study the stabiliser of the
set $V$ and later derive the extended symmetries of the induced
dynamical system. This provides an interesting connection between an
algebraic and a dynamical property, here via the connection between
the normaliser, the unit group $\cO^{\times}$, and the Galois group of
$K/\QQ$.  Later, when we consider the induced shift spaces more
closely, another connection of this kind shows up, then between
topological entropy and the values of Dedekind zeta functions at
integer values.  \smallskip

The paper is organised as follows. In Section~\ref{sec:initial}, we
set the scene with some initial examples of quadratic fields, where we
determine the stabiliser for the $k$-free integers in
$\QQ (\sqrt{-2\ts}\,)$ and recall previous results from \cite{BBHLN}.
Then, Section~\ref{sec:quadratic} covers the case of all quadratic
fields, where the special cases treated before will come in handy as
they turn out to be the ones that indeed need special treatment.

Afterwards, in Section~\ref{sec:shift}, we construct the shift spaces
that emerge as the orbit closure of the $k$-free points under the
lattice translation action in the Minkowski embedding. Here, we use a
special variant of the embedding such that all systems are acted on by
the same group, namely the integer lattice $\ZZ^2$.  This allows to
determine the centraliser and the normaliser of the $k$-free shifts
for arbitrary quadratic fields in a unified way.

Next, we address the question of how to classify the shifts up to
topological conjugacy. While some distinctions are possible on the
basis of the normaliser, we need topological entropy for a finer
distinction, as we discuss in Section~\ref{sec:entropy}. Here, the
entropy is expressed in terms of special values of the Dedekind zeta
function, thus providing another link between an algebraic and a
dynamical quantity. We also refine our viewpoint by considering factor
maps, which define semi-conjugacies, some of which can then be
excluded as well.

\section{Initial examples of quadratic fields}\label{sec:initial}

Let us begin with the example of the imaginary quadratic field
$K=\QQ (\sqrt{-2\ts}\,)$, which has ring of integers
\[
  \cO \, = \, \cO^{}_{\! K} \, = \, \ZZ[ \sqrt{-2\ts} \, ] \ts .
\]
Its unit group is the smallest one possible, which is to say that
$\cO^{\times} \nts\nts = \{ \pm 1 \}\simeq C_2$. Here and below, $C_n$
denotes the cyclic group of order $n$.  The Galois group is another
$C_2$, with complex conjugation (denoted by $\bar{.}\,$) as the
non-trivial automorphism, and the field norm is $N(x) = x \bar{x}$.
Since $\cO$ has class number $1$, we can use numbers (rather than
ideals) for this initial example. For an integer $k\geqslant 2$, we
say that $x\in\cO$ is \emph{$k$-free} if it is not divisible by the
$k$-th power of any prime in $\cO$.

\begin{prop}\label{prop:min-two}
  Let\/ $2 \leqslant k \in \NN$ be fixed and consider the set\/
  $V\nts\nts =V^{}_{\nts k}$ of\/ $k$-free integers in\/
  $\cO = \ZZ[\sqrt{-2\ts}\,]$.  Let\/ $A$ be a\/ $\ZZ$-linear
  bijection of\/ $\cO$ with\/ $A(V) \subseteq V\!$. Then, $A$ is of
  the form\/ $A(x) = \varepsilon \ts\ts \sigma (x)$ with\/
  $\varepsilon \in \cO^{\times}\simeq C_2$ and\/
  $\sigma \in \{ \id, \bar{.} \, \}\simeq C_2$.  Consequently,
  $A(V)=V\!$, and these mappings form the group\/
  $\ts\stab (V) \simeq C_2 \times C_2$.
\end{prop}

\begin{proof}
  Let $A$ be a $\ZZ$-linear bijection of $\cO$ with
  $A(V) \subseteq V\!$. If $x\in V$ is coprime with a rational prime
  $p$ that is unramified, so $\gcd^{}_{\cO} (x,p) = 1$ where the
  $\gcd^{}_{\cO}$ in $\cO$ is unique up to units, we know that
  $p^{\ell} x \in V$ for every $1 \leqslant \ell < k$, hence also
  $A (p^{k-1} x) = p^{k-1} A(x) \in V \!$, which implies
  $\gcd^{}_{\cO} ( A(x), p) = 1$. Since no odd rational prime is
  ramified in $K$, this observation provides a powerful coprimality
  structure.
  
  Let $U \nts =\cO^{\times} =\{ \pm 1\}$ and set $\xi=\sqrt{-2\ts}$.
  For $p=2=-\xi^2$, which is the only ramified prime in this case, we
  may now conclude from the above coprimality structure that
\[
   A (U) \, \subseteq \, U \cup \xi\ts\ts U 
    \cup \dots \cup \xi^{k-1} U ,
\]
which we will now reduce to $A(U) \subseteq U\!$, and thus to $A(U)=U$
since $U$ is finite.  Without loss of generality, we may assume
$A(1) = \xi^m$ for some $0\leqslant m < k$, possibly after replacing
$A$ by $-A$, which is a map of the same kind.  Also, we know that
$A(\xi) = a + b \ts \xi$ for some $a,b \in \ZZ$.

If we compute $\det(A)$ with respect to the basis $\{ 1, \xi \}$, we
get $\det (A) = (-2)^r b$ when $m=2r$ and $\det(A) = - (-2)^r a$
when $m=2r+1$.  Since $A$ is bijective on $\cO$, it is unimodular as
an integer matrix, so $\det (A) = \pm 1$. This forces $r=0$, and we
either get $m=0$, together with $b=\pm 1$, or $m=1$, then with
$a=\pm 1$. If $m=1$, we thus have $A(\xi) = \pm 1 + b\ts \xi$, which
has norm $N(\pm 1 + b \ts \xi) = 1 + 2 b^2$.  This must be a
  power of $2$ by coprimality, which is only possible for $b=0$. But
$A(1)=\xi$ and $A(\xi)=\pm 1$ implies $A(3+\xi) = \pm 1 +3 \xi$, thus
mapping an element of $V$ of norm $11$ to an image of norm $19$, which
is impossible by the coprimality structure. This rules out $m=1$.

Finally, if $m=0$ and $b=\pm 1$, we get $A(\xi) = a \pm \xi$ with
$a\in\ZZ$ and $N(A(\xi))= a^2+2$, which is a power of $2$ only for
$a=0$. This leads to $A(1) = 1$ together with $A(\xi)=\pm \xi$, or to
$-A$, which are the four elements of the form
$A(x) = \varepsilon \ts\ts \sigma (x)$ stated in the proposition. They
clearly map units to units, and all remaining claims are clear.
\end{proof}

Here and below, given a set $U\!$, the notation $\stab (U)$ refers to
the monoid of $\ZZ$-linear mappings that send $U$ \emph{into} itself.
It is thus part of the above result that the stabiliser of
$V^{}_{\nts k} \subset \ZZ[\sqrt{-2\ts}\,]$ is actually a group.  In
\cite{BBHLN}, the corresponding result was proved for the imaginary
quadratic fields $\QQ (\sqrt{d \ts}\, )$ with $d\in \{ -1 , -3 \}$,
which are statements about the $k$-free elements of the Gaussian and
the Eisenstein ring of integers, $\ZZ[\ii]$ and $\ZZ[\rho]$ with
$\rho = \frac{1}{2} ( -1 + \ii \sqrt{3}\,)$.  Let us first recall the
Gaussian case from \cite[Lemma~6.1]{BBHLN}. Here and below,
$D_n = C_n \nts\rtimes C_2$ denotes the dihedral group of order $2n$.

\begin{fact}\label{fact:Gauss}
  Let\/ $V^{}_{\nts k}$ be the set of $k$-free Gaussian integers, for
  some fixed\/ $2\leqslant k \in \NN$. Then, any\/ $\ZZ$-linear
  bijection\/ $A$ of\/ $\ZZ [\ii]$ that satisfies\/
  $A ( V^{}_{\nts k}) \subseteq V^{}_{\nts k}$ is of the form\/
  $A(x) = \varepsilon \ts\ts \sigma(x)$ with\/
  $\varepsilon \in \ZZ[\ii]^{\times} \nts = \{ 1, \ii, -1, -\ii \}
  \simeq C_4$ and\/ $\sigma \in \{ \id, \bar{.}\, \} \simeq C_2$.
   
  These mappings are bijections of\/ $V^{}_{\nts k}$ and form the
  group\/ $\ts \stab (V^{}_{\nts k}) \simeq C_4 \rtimes C_2 = D_4$,
  which is a maximal finite subgroup of\/ $\GL (2,\ZZ)$, and
  independent of\/ $k$.  \qed
\end{fact}

The analogous statement for the Eisenstein integers reads as follows;
see \cite[Thm.~6.5]{BBHLN}.

\begin{fact}\label{fact:Eisen}
  Let\/ $V^{}_{\nts k}$ be the set of $k$-free Eisenstein integers,
  for some fixed\/ $2\leqslant k \in \NN$. Then, any\/ $\ZZ$-linear
  bijection\/ $A$ of\/ $\ZZ [\rho]$ that satisfies\/
  $A ( V^{}_{\nts k} ) \subseteq V^{}_{\nts k}$ is of the form\/
  $A(x) = \varepsilon \ts\ts \sigma(x)$ with\/
  $\varepsilon \in \ZZ[\rho]^{\times} \nts = \{ (-\rho)^m: 0\leqslant
  m \leqslant 5 \} \simeq C_6$ and\/
  $\sigma \in \{ \id, \bar{.}\, \} \simeq C_2$.
   
  These mappings are bijections of\/ $V^{}_{\nts k}$ and form the
  group\/ $\ts \stab (V^{}_{\nts k}) \simeq C_6 \rtimes C_2 = D_6$,
  which is another maximal finite subgroup of\/ $\GL (2,\ZZ)$, again
  independent of\/ $k$.  \qed
\end{fact}

Also, the stabiliser was determined for some real quadratic fields,
namely for $\QQ (\sqrt{d\ts}\,)$ with $d \in \{ 2, 3, 5 \}$. The
corresponding results from \cite[Sec.~7]{BBHLN} can be summarised as
follows, where $C_{\infty}$ denotes the infinite cyclic group.

\begin{fact}\label{fact:real-ex}
  Consider\/ $K = \QQ (\sqrt{d\ts}\,)$ for fixed\/ $d \in \{ 2,3,5\}$,
  and let\/ $V^{}_{\nts k}$ be the set of\/ $k$-free integers in\/
  $\cO^{}_{\nts K}$, that is, in\/ $\ZZ[\sqrt{2}\,]$, in\/
  $\ZZ[\sqrt{3}\,]$, or in\/ $\ZZ[\tau]$ with\/
  $\tau = \frac{1}{2} (1+\sqrt{5}\,)$. Then, the\/ $\ZZ$-linear
  bijections\/ $A$ of\/ $\cO^{}_{\nts K}$ with\/
  $A(V^{}_{\nts k}) \subseteq V^{}_{\nts k}$ are precisely the
  mappings\/ $A(x) = \varepsilon\ts\ts \sigma(x)$ with\/
  $\varepsilon \in \cO^{\times}_{\nts K} \simeq C_2 \times C_{\infty}$
  and\/ $\sigma \in \{ \id, (.)' \ts \}$, where\/ $(.)'$ denotes
  algebraic conjugation in\/ $K$.
   
  These mappings are bijections of\/ $V^{}_{\nts k}$ and form the
  group\/
  $\ts \stab (V^{}_{\nts k}) = \cO^{\times}_{\nts K} \rtimes C_2
  \simeq C_2 \times D_{\infty}$, which is a proper infinite subgroup of\/
  $\GL (2,\ZZ)$ that does not depend on\/ $k$.  \qed
\end{fact}

All examples so far have class number $1$. The natural next step is to
extend the analysis to \emph{all} quadratic fields, where key notions
have to formulated via ideals. In this process, the above examples
will emerge as cases that need special treatment, in one way or
another.

\section{General quadratic fields}\label{sec:quadratic}

Let us recall some notation and basic results on quadratic fields, all
of which can be found in \cite{Zagier}. We use the standard
parameterisation by a square{\ts}-free integer
$d\in \ZZ \setminus \{ 0,1 \}$ and consider $K=\QQ (\sqrt{d\ts}\,)$.
The field discriminant $d^{}_{\nts K}$ and the ring of integers
$\cO^{}_{\nts K}$ (which is the maximal order of $K$) are given by
\[
  d^{}_{\nts K} \, = \, \begin{cases}  d , &
    \text{if } d\equiv 1 \bmod 4 , \\
          4 \ts d , & \text{if } d \equiv 2,3 \bmod 4 ,
          \end{cases}
  \quad \text{and} \quad
    \cO^{}_{\nts K} \, = \, \begin{cases} 
       \ZZ \oplus \ZZ \frac{1+\sqrt{d\ts}}{2}  , & 
         \text{if } d\equiv 1 \bmod 4 , \\
         \ZZ \oplus \ZZ \sqrt{d\ts} ,
         & \text{if } d \equiv 2,3 \bmod 4 .
          \end{cases}
\]
The Galois group is $\Gal (K/\QQ) = \{ \id, (.)' \} \simeq C_2$, where
$(.)'$ denotes algebraic conjugation in $K$, as induced by
$\sqrt{d\ts} \mapsto - \sqrt{d\ts}$. This simply is complex
conjugation for all imaginary quadratic fields. The field norm is
given by $N(x) = x \ts x'$, which can be negative for real fields.

Let us recall Dirichlet's unit theorem for quadratic fields as follows.
\begin{fact}\label{fact:units}
  The unit group\/ $\cO^{\times}_{\nts K}$ for imaginary quadratic
  fields is always finite, and isomorphic with\/ $C_4$ for\/ $d=-1$,
  with\/ $C_6$ for\/ $d=-3$, and with\/ $C_2$ for all remaining\/ $d<0$.
    
  For all real quadratic fields, the unit group is infinite,
  $\cO^{\times}_{\nts K} \simeq C_2 \times C_{\infty}$.  \qed
\end{fact}

For $x\ne 0$, the principal ideal $(x)$ generated by $x$ has the
unique decomposition
\[
    (x) \, = \, \prod_{\fp} \fp^{v^{}_{\fp} (x)}
\]
into powers of prime ideals, where the product runs over all prime
ideals of $\cO^{}_{\nts K}$, but the valuation $v^{}_{\fp} (x)$
vanishes for all but (at most) finitely many of them. For quadratic
fields, a rational prime $p$ is \emph{ramified} if and only if
$p \ts\ts | \ts d^{}_{\nts K}$. In particular, one has
\begin{equation}\label{eq:ram-prod}
     \bigl( \sqrt{d\ts}\,\bigr) \, = \, \prod_{\fp | d} \fp \ts .
\end{equation}

\begin{definition}
  Let $ 2 \leqslant k \in \NN$ be fixed. An element
  $0\ne x\in \cO^{}_{\nts K}$ is called \emph{$k$-free} if
  $v^{}_{\fp} (x) < k$ for all prime ideals $\fp$ in
  $\cO^{}_{\nts K}$. The set of all $k$-free integers is denoted by
  $V^{}_{\nts k}$.
\end{definition}

When the class number is $1$, which is to say that all ideals in $\cO$
are principal, our above definition agrees with the traditional one
that $x$ is $k$-free if it is not divisible by the $k$-th power of a
prime in $\cO$. However, beyond the class number $1$ situation, one
needs to employ ideals and valuations as above.  By definition, $0$ is
never an element of $V^{}_{\nts k}$, and the following result
justifies why the study of $k=2$ plays a special role, for any fixed
square{\ts}-free $d\in \ZZ \setminus \{ 0,1 \}$.

\begin{lemma}\label{lem:2-suffices}
  Let\/ $A$ be a\/ $\ZZ$-linear bijection of\/ $\cO = \cO^{}_{\nts K}$
  that satisfies\/ $A (V^{}_{\nts k} ) \subseteq V^{}_{\nts k}$ for
  some integer\/ $k \geqslant 2$.  Then, one also has\/
  $A (V^{}_{2} ) \subseteq V^{}_{2}$.
\end{lemma}

\begin{proof}
  Assume $A ( V^{}_{\nts k} ) \subseteq V^{}_{\nts k}$, let
  $x\in V^{}_{2}$ and consider any non-ramified (rational) prime $p$,
  so $p \nts\nmid\! d^{}_{\nts K}$. Then,
  $p^{k-2} x \in V^{}_{\nts k}$, and $\ZZ$-linearity of $A$ implies
\[
     p^{k-2} A(x) \, = \, A(p^{k-2} x) \, \in \, V^{}_{\nts k} \ts ,
\]
thus giving $v^{}_{\fp} \bigl( A(x)\bigr) \leqslant 1$ for all
$\fp \ts | \ts \ts p$.

It remains to consider the ramified primes, that is, primes with
$p \ts \ts | \ts d^{}_{\nts K}$.  For any such prime, we have
$(p) = p \ts\ts \cO = \fp^2$, with $\fp$ denoting the corresponding
prime ideal over $p$ in $\cO$.  When $k$ is even and $x\in V^{}_{2}$,
one has
\[
    v^{}_{\fp} \bigl( p^{\frac{k}{2} - 1} x \bigr) \, = \,
    v^{}_{\fp} \bigl( p^{\frac{k}{2}-1}\bigr) + v^{}_{\fp} (x)
    \, = \,  k - 2 + v^{}_{\fp} (x) \, \leqslant \, k-1 \, < \, k \ts ,
\]
so $p^{\frac{k}{2} - 1} x \in V^{}_{\nts k}$, hence also
$A \bigl( p^{\frac{k}{2} - 1} x \bigr) = p^{\frac{k}{2} - 1} A(x) \in
V^{}_{\nts k}$. With
$v^{}_{\fp} \bigl( p^{\frac{k}{2} - 1}\bigr) = k-2$, this implies
$v^{}_{\fp} \bigl( A(x) \bigr) \leqslant 1$, and we get
$A ( V^{}_{2} ) \subseteq V^{}_{2}$ for $k$ even.

Now, consider $k$ odd and $x\in V^{}_{2}$, so
$v^{}_{\fp} (x) \in \{ 0,1 \}$. Here, we have
\[
   v^{}_{\fp} \bigl( p^{\frac{k-1}{2} - 
     v^{}_{\fp} (x)} \bigr) \, = \, k-1 - 2 \ts v^{}_{\fp} (x)
     \, < \, k - 2 \ts v^{}_{\fp} (x) \ts ,
\]
and thus $p^{\frac{k-1}{2} - v^{}_{\fp} (x)} x \in V^{}_{k}$.
Now, $p^{\frac{k-1}{2} - v^{}_{\fp} (x)} A(x) =
A \bigl( p^{\frac{k-1}{2} - v^{}_{\fp} (x)} x \bigr) \in V^{}_{k}$ gives
\begin{equation}\label{eq:max-2}
    v^{}_{\fp} (A(x)) \, \leqslant \, 2 \ts v^{}_{\fp} (x) \ts .
\end{equation}
When $v^{}_{\fp} (x) = 0$, we also get $v^{}_{\fp} (A(x)) =0 $, and we
are good in this case, too. Next, consider $v^{}_{\fp} (x) = 1$, where
$v^{}_{\fp} \bigl( A(x) \bigr) \leqslant 2$.  From
\eqref{eq:ram-prod}, if $\fp \ts | \ts\ts d$, we see that
$v^{}_{\fp} \bigl( x \sqrt{d\ts} \, \bigr) = 2$. As $\fp^2 = (p)$, we
have $x\sqrt{d\ts} = p \ts\ts y$ for some $y\in\cO$, hence
$A \bigl( x \sqrt{d\ts} \, \bigr) = p A(y) \in p \ts \cO = (p)$.
Suppose $v^{}_{\fp} \bigl( A(x)\bigr) \geqslant 2$, which means
$v^{}_{\fp} \bigl( A(x)\bigr) = 2$ by \eqref{eq:max-2}.  Then, we also
have $A(x) \in p \ts \cO$ and
\[
     A \bigl( x \ts (\ZZ \oplus \ZZ \sqrt{d\ts}\,) \bigr) \, = \:
      \ZZ A(x) \oplus \ZZ A ( x \sqrt{d\ts}\,) \ts . 
\]
Consequently, $p^2$ divides the index
$\bigl[\cO : A \bigl( x \ts \ZZ [\sqrt{d\ts}\,] \bigr) \bigr]$.

When $d \equiv 1 \bmod 4$, we know that $2$ is not ramified, so any
ramified prime $p$ must be odd. Here, we have
$\bigl[ \cO : \ZZ[\sqrt{d\ts}\,] \bigr] = 2$, hence also
$\bigl[ A(x \ts \cO) : A \bigl( x \ts \ZZ[\sqrt{d\ts} \, ]
\bigr)\bigr] = 2$, which implies 
\[
    p^2   \: \big| \:
    [\cO : A (x\ts \cO) ]  \, = \, [ \cO : x \ts \cO]
    \, = \, \No (x \ts \cO) \ts ,
\]
where $\No (\fb) \defeq [\cO : \fb]$ denotes the \emph{absolute} norm
of an ideal $\fb$ in $\cO$. When $\fb$ is a principal ideal, hence
$\fb = (z) = z \ts \cO$ for some $z\in\cO$, the absolute norm is
related to the field norm by $\No (\fb) = \lvert N(z) \rvert$, so we
get $p^2 \ts\ts | \ts\ts N(x)$.

When $d\equiv 2,3 \bmod 4$, the condition $p^2 \ts\ts | \ts\ts N (x)$
follows directly, without the extra index-$2$ argument. In fact, it
then also applies to $p=2$ for $d\equiv 2$. On the other hand, since
$v^{}_{\fp} (x) = 1$, we know that $N (x)$ is exactly divisible by
$p$, hence not by $p^2$. This contradiction shows that
$v^{}_{\fp} \bigl( A (x) \bigr) \geqslant 2$ is impossible for any odd
ramified prime, as well as for $p=2$ when $d\equiv 2 \bmod 4$.

Finally, when $d\equiv 3 \bmod 4$, the prime $2$ is ramified, and we
need to check what happens with the corresponding prime ideal
$\fp^{}_{2}$, where $\fp^2_2 = 2 \ts \cO$. Here,
$\cO = \ZZ [ \sqrt{d\ts}\,]$, and one has
\[
  x^2 - d \, \equiv \, x^2 - 1 \, \equiv \,
  (x-1)^2 \quad \bmod  2 \ts .
\]
Invoking \cite[Thm.~I.8.3]{N}, we get that
$\fp^{}_{2} = (2, \sqrt{d\ts} - 1)$, which comprises
$\sqrt{d\ts} - 1 \in \fp^{}_{2}$, so that
$v^{}_{\fp^{}_2} (\sqrt{d\ts} - 1) \geqslant 1$.  Since
$v^{}_{\fp^{}_2} (2) = 2$, we then actually have
$v^{}_{\fp^{}_2} (\sqrt{d\ts} - 1) = 1$, as we would otherwise get
$ (2, \sqrt{d\ts}-1) \subseteq \fp^2_2$ and hence a contradiction.
Now, $v^{}_{\fp^{}_{2}} (x) = 1$ implies
$v^{}_{\fp^{}_{2}} \bigl( x \ts (\sqrt{d\ts}-1)\bigr) =2$, and thus
$x \ts (\sqrt{d\ts}-1) = 2 y$ for some $y\in \cO$.

Now, suppose $v^{}_{\fp^{}_{2}} \bigl( A(x)\bigr) \geqslant 2$, hence
$v^{}_{\fp^{}_{2}} \bigl( A(x)\bigr) = 2$ by \eqref{eq:max-2}
again. As we now have $\cO = \ZZ \oplus \ZZ (\sqrt{d\ts}-1)$, we see
that
$A(x\ts \cO) = \ZZ A(x) \oplus \ZZ A \bigl( x (\sqrt{d\ts}-1) \bigr)
\subseteq 2 \ts \cO$. This gives
\[
    4 \: \big| \: [\cO : A (x\cO)] \, = \, [\cO : x \ts \cO]
    \, = \, \No (x\cO) \, = \, \lvert N (x) \rvert \ts ,
\]
but $4$ cannot divide $N (x)$ for $x \in V^{}_{2}$, whence
$v^{}_{\fp^{}_{2}} \bigl( A(x)\bigr) \geqslant 2$ is impossible
as well.
\end{proof}

In view of Lemma~\ref{lem:2-suffices}, we now concentrate our
attention to the case $k=2$, and set $V\nts\nts = V^{}_{2}$.

\begin{definition}
  Let $K$ be a quadratic field, with ring of integers
  $\cO = \cO^{}_{\nts K}$ and $V\nts\subset \cO$ the subset of
  square{\ts}-free elements. A $\ZZ$-linear bijection $A$ of $\cO$ is
  called a \emph{preserving map} (PM) for $V$ if $A(V)\subseteq V\!$,
  and a \emph{strongly preserving map} (SPM) if $A(V) = V\!$.

  The set of all PMs for $V$ constitute the \emph{stabiliser} 
  of $V\!$, denoted by $\stab (V)$. 
\end{definition}

Note that the restriction of a PM to $V$ is obviously injective, but
it is not clear a priori whether it is also surjective. Since the
$\ZZ$-linear bijections of $\cO$ form a group, the subset of PMs for
$V$ inherits a semi-group structure with a unit, which is to say that
$\stab (V)$ is a monoid. However, whether or when $\stab (V)$ is a
group remains to be determined.

Let us first look at the coprimality structure within $V\!$. We say
that two elements $x,y \in \cO$ are \emph{coprime}, denoted by
$(x,y)=1$, if the principal ideals $(x)$ and $(y)$ have disjoint
decompositions into prime ideals of $\cO$.

\begin{lemma}\label{lem:coprime}
  Let\/ $p$ be a rational prime with\/
  $p \nts \nmid \nts\nts d^{}_{\nts K}$, and let\/ $x\in V$ be coprime
  with\/ $p$, so\/ $(p,x)=1$.  Then, for any\/ $A \in \stab (V)$, one
  also has\/ $(p, A(x))=1$.
\end{lemma}

\begin{proof}
  By assumption, $p$ is not ramified, and thus not a square, so the
  condition $x\in V$ together with $(p,x)=1$ implies $p\ts x\in V\!$,
  hence also $p A(x) = A(p\ts x) \in V$ due to $A(V) \subseteq V$
  together with $\ZZ$-linearity.  But this is only possible if
  $(p,A(x))=1$ as claimed.
\end{proof}

In generalisation of previous arguments, we now need to
consider the set
\[
  W \, \defeq \, \{ x \in V : v^{}_{\fp} (x) > 0
  \;  \Rightarrow \;
  \fp \ts\ts | \ts\ts (d^{}_{\nts K}) \} \ts .
\]
The norms of elements in $W$ have a prime decomposition into ramified
primes only.  In particular, given $x\in W$ and any prime ideal $\fp$
over a ramified $p$, one has $v^{}_{\fp} (x) \in \{ 0,1 \}$, while all
other valuations vanish.

\begin{lemma}\label{lem:W-to-W}
  If $A$ is a\/ \textnormal{PM} for\/ $V\!$, it
  satisfies\/ $A(W) = W\!$.
\end{lemma}

\begin{proof}
  Due to the coprimality structure stated in Lemma~\ref{lem:coprime},
  it is clear that $A(W) \subseteq W\!$. When $K$ is imaginary
  quadratic, $W$ is a finite set, and bijectivity of $A$ then implies
  $A(W)=W\!$.  So, it remains to consider the case that $K$ is real
  quadratic, where the unit group, and then also $W\!$, is an infinite
  set.

  The norm on $W$ takes only finitely many distinct values. This is so
  because there are only finitely many ramified primes that can show
  up in the prime decomposition of $N(x)$ for any $x\in W\!$, with
  powers $0$ or $1$, while units have norm $\pm 1$.  Let $C$ be the
  set of these values and
  $S^{\ts\prime}_c \defeq \{ x\in W : N(x) = c \}$, so that
  $W=\bigcup_{c\ts\in C} S^{\ts\prime}_c$ is a finite union of
  disjoint sets, where each $S^{\ts\prime}_c$ itself is infinite,
   because the unit group is already infinite.
  
  Let us write $\cO = \ZZ \oplus \ZZ \delta = \ZZ[\delta]$ with
\begin{equation}\label{eq:basis}
     \delta \, = \,
     \begin{cases} \sqrt{d\ts} , & d\equiv 2,3 \bmod 4 \ts , \\
         \frac{1}{2} (1 + \sqrt{d\ts}\,), & d\equiv 1 \bmod 4  \ts .
         \end{cases}
\end{equation}
  With this choice of an integral basis for $\cO$, we have
  $x = a + b \delta$ with $a,b \in \ZZ$ for any $x\in \cO$, and the
  field norm is $N(x) = Q(a,b)$, where $Q$ is a non-degenerate
  quadratic form in $a$ and $b$. Thus, we have
\begin{equation} \label{eqn:S_c}
  S_c \, \defeq \, \{ (a,b) \in \ZZ^2 : a + b\ts \delta \in
  S^{\ts\prime}_c \} \, = \, \{ (a,b) \in \ZZ^2 : Q(a,b) = c \} \ts ,
\end{equation}
where the equality follows because any element $a+b\ts \delta$ with
$Q(a,b) = c$ is square{\ts}-free by construction.

Written in the $\ZZ$-basis $\{ 1, \delta\}$, our map $A$ is
represented by a $\GL (2,\ZZ)$-matrix, also called $A$ for simplicity,
which acts linearly on all of $\RR^2$.  We thus see that $S_c$ is the
intersection of a quadratic curve (or conic)
$\widehat{S}_c \subset \RR^2$ with $\ZZ^2$. Since $A$ maps
$\bigcup_{c\ts\in C} S_c$ into itself, any point from $S_c$ must be
mapped to a point from $S_{c^{\ts\prime}}$ for some
$c^{\ts\prime}\in C$. Clearly, $C$ is finite but $S_c$ is not,
  because $(a,b) \leftrightarrow a + b\ts \delta$ is a bijection between
  $S_c$ and $S'_{c}$, where the latter is infinite. Now, Dirichlet's
pigeon hole principle implies that, for some power $A^n$ of $A$, there
exists a $c^{}_{1}\in C$ such that
$S_{c^{}_{1}} \cap A^n(S_{c^{}_{1}})$ is an \emph{infinite} set, so
also $\widehat{S}_{c^{}_{1}} \cap A^n(\widehat{S}_{c^{}_{1}})$ is
infinite.

Now, $A^n (\widehat{S}_{c^{}_{1}})$ is a non-degenerate conic
as well, because it is the image of one under a linear
bijection. Since non-degenerate conics cannot intersect in infinitely
many points unless these conics are equal, because $5$ points 
determine a conic, see \cite[Sec.~14.7]{Cox},  we get
$\widehat{S}_{c^{}_{1}} = A^n(\widehat{S}_{c^{}_{1}})$. Now, we have
\[
  A^n (S_{c^{}_{1}}) \, = \,
  A^n \bigl(\widehat{S}_{c^{}_{1}} \cap \ZZ^2 \bigr)
  \, = \, A^n (\widehat{S}_{c^{}_{1}}) \cap A^n (\ZZ^2) \, = \,
  \widehat{S}_{c^{}_{1}} \cap \ZZ^2 \, = \, S_{c^{}_{1}} \ts ,
\]
which also implies that $A^n$ maps
$\bigcup_{c \in C\setminus \{c^{}_{1}\}} S_{c}$ into itself.

At this point, we can repeat the argument for the smaller union, where
we get some power of $A^n$ that maps some $S_{c^{\ts\prime}}$ into
itself. After finitely many steps, a single $S_{c^{\ts\prime\prime}}$
remains, which is then automatically invariant, by a simplified
argument of the above type.  So, we see that some power of $A$, say
$A^m$, satisfies $A^m (S_c) = S_c$ for all $c \in C$ and thus maps $W$
onto itself. Now, if $A(W)$ were a strict subset of $W\!$, this would
imply
\[
  W \, = \, A^m (W) \, = \, A^{m-1} (A(W))
  \, \subsetneq \, A^{m-1} (W) \, \subseteq \, W ,
\]
which is a contradiction, so we also get $A(W) = W$ as claimed.
\end{proof}

Since algebraic conjugation maps $W$ onto itself, the following
consequence is immediate.

\begin{fact}\label{fact:Gal-W}
  One has\/ $\sigma (W)  = W$ for all\/ $\ts\sigma \in
  \Gal (K/\QQ)$.   \qed
\end{fact}

 We will also make use of the following easy observation.

\begin{fact}\label{fact:square-free-divisor}
  Let\/ $x \in W\!$. Then, $N(x)$ is a square-free divisor
  of\/ $d_{K}$. \qed
\end{fact}

If $\varepsilon$ is a unit in $\cO$, we use $m^{}_{\varepsilon}$ to
denote the mapping $x \mapsto m^{}_{\varepsilon} (x) \defeq
\varepsilon \ts x$.

\begin{prop}\label{prop:unit-suffices}
  Let\/ $K$ be a quadratic field and\/ $\cO = \cO^{}_{\! K}$ its
  ring of integers.  Then, for any\/ $A \in \stab (V)$, the following
  statements are equivalent.
  \begin{enumerate}\itemsep=2pt
  \item $A=m^{}_{\varepsilon} \nts\circ \sigma$ for some\/
     $\varepsilon \in \cO^{\times}$ and\/ $\sigma \in
     \Gal (K/\QQ)$.
  \item $A(1) \in \cO^{\times}\!$.
\end{enumerate}
\end{prop}
  
\begin{proof}
  The implication $(1) \Rightarrow (2)$ is clear. To show the converse
  direction, we observe that $A(1) = \varepsilon \in \cO^{\times}$
  implies that $A' = m^{}_{\varepsilon^{-1}} \circ A$ is a mapping of
  the same type, wherefore we may assume $A(1) = 1$ without loss of
  generality. Now, we have to distinguish two cases.\smallskip
    
\textit{Case} $1$: $d \equiv 1 \bmod 4$, so $\cO = \ZZ[\delta]$ with
  $\delta = \frac{1 + \sqrt{d\ts}}{2}$. Let
  $A(\delta) = a + b \ts \delta$ with $a,b \in \ZZ$. Since $A$ is a
  bijection of $\cO$, $A(1)=1$ and $A(\delta)$ must generate $\cO$ as
  a $\ZZ$-module, so $A$ is represented by the matrix
  $\left(\begin{smallmatrix} 1 & a \\ 0 & b \end{smallmatrix} \right)
  \in \GL (2,\ZZ)$, hence $b = \det (A) = \pm 1$. Set
  $\tilde{a}= a + \frac{1}{2} (b-1)$.

  Now, we have
  $A(\sqrt{d\ts}\,) = A(2 \delta - 1) = 2 \ts a - 1 \pm (1 +
  \sqrt{d\ts}\,) = 2 \ts \tilde{a} \pm \sqrt{d\ts}$, which has norm
  $4 \ts \tilde{a}^2 - d$. Since
  ${\sqrt{d\ts}{\ts\ts}}^2 = d = d^{}_{\nts K}$ in this case, we have
  $(\sqrt{d\ts}\,) = \prod_{p \ts | \ts d^{}_{\nts K}} \fp$, where
  $(p) = \fp^2$ for each factor, which implies $\sqrt{d\ts} \in W\!$,
  hence also $A(\sqrt{d\ts}\,) \in W$ by Lemma~\ref{lem:W-to-W}.
  Moreover, the norm of $A(\sqrt{d\ts}\,)$ must be a divisor of
  $d=d^{}_{\nts K}$ by Fact~\ref{fact:square-free-divisor}.

  We now claim that $2 \ts n \ts \tilde{a} \pm \sqrt{d\ts} \in W$ for
  all $n\in \NN_0$. Indeed, since $\pm \sqrt{d\ts} \in W\!$, this is
  clear for $n=0$. Assuming the claim to hold for $n$, we have
  $x_n = 2 \ts n \ts \tilde{a} \pm \sqrt{d\ts} \in W\!$, and also its
  image must be in $W\!$, again by Lemma~\ref{lem:W-to-W}. But this
  means that
  $A(x_n) = 2\ts (n+1) \ts \tilde{a} \pm \sqrt{d\ts} \in W\!$, for one
  of the signs, and then actually for both of them, by an application
  of Fact~\ref{fact:Gal-W}. This settles the claim inductively.

  As a result, we see that $N (2\ts n \ts\tilde{a} + \sqrt{d\ts}\,)$
  divides $d$ for all $n\in\NN$, where the norm is
  $4 \ts n^2 \tilde{a}^2 - d$.  Since this is unbounded unless
  $\tilde{a}=0$, we may conclude that
  $A (\sqrt{d\ts}\,) = \pm \sqrt{d\ts}$. Consequently, when $b=1$, we
  get $a=0$ and $A=\id$, while $b=-1$ forces $a=1$, which gives
  $A(\delta) = \delta^{\ts\prime}$ with $(.)' \in \Gal(K/\QQ)$ being
  algebraic conjugation. This settles the proposition for Case
  $1$. \smallskip

\textit{Case} $2$: $d \equiv 2,3 \bmod 4$, so
  $\cO = \ZZ \oplus \ZZ \sqrt{d\ts}$. Let
  $A(\sqrt{d\ts}\,) = a+b \ts \sqrt{d\ts}$ with $a,b \in \ZZ$. Here,
  $A(1)=1$ and $A(\sqrt{d\ts}\,)$ generate $\cO$, so $b = \det \left(
    \begin{smallmatrix} 1 & a \\ 0 & b \end{smallmatrix} \right) =\pm
  1$, because $A$ is a bijection of $\cO$. In complete analogy to the
  first case, one shows by induction that
  $n \ts a + \sqrt{d\ts} \in W$ for all $n\in\NN_0$. Since
  $n \ts a + \sqrt{d}$ has norm $n^2 a^2 - d$, but divides
  $2 \, \lvert d \rvert$ by Fact~\ref{fact:square-free-divisor},
  a contradiction is only avoided if $a=0$, which gives
  $A (\sqrt{d\ts}\,) = \pm \sqrt{d\ts}$, and we are done.
\end{proof}
  
To continue, we shall need the following property of the splitting
primes.
  
\begin{fact}\label{fact:density}
   For any quadratic number field\/ $K$, there are infinitely many
  rational primes\/ $q$ that split in\/ $K$ in such a way that the
  prime ideals over\/ $q$ are principal.
\end{fact} 

\begin{proof}
    This is a consequence of Dirichlet's density theorem
    \cite[Thm.~VII.13.2]{N} as follows.  Let $M$ be a set of prime
    ideals of $K$.  As explained in \cite[p.~543]{N}, the Dirichlet
    density of $M$ is given by
\[
  d(M) \, = \, \lim_{s\searrow 1}
  \frac{\sum_{\fq \in M}N(\fq)^{-s}}{\log \frac{1}{s-1}} \ts .
\]
Let $S_{\mathrm{ram}}$ be the set of all rational primes which are
ramified in $K$, and let $S_{\mathrm{ram}}(K)$ be the set of primes
(meaning prime ideals) in $K$ above those in $S_{\mathrm{ram}}$.  We
recall that $N(\fq) = q^f$, where $q$ is the rational prime below
$\fq$ and $f$ is called the degree of $\fq$, which is equal to $2$ if
$q$ is inert in $K$ and equal to $1$ otherwise.  As the sum
$\sum N(\fq)^{-s}$ over all prime ideals in $M$ of degree $2$
converges and the set $S_{\mathrm{ram}}(K)$ is finite, the definition
of $d(M)$ only depends upon primes $\fq \not\in S_{\mathrm{ram}}(K)$
of degree $1$.  So if $M$ has positive density, there are infinitely
many primes $\fq \in M$ such that the rational prime $q$ below $\fq$
splits in $K$.
  
We now apply this observation to the set of all primes which are
principal, which has positive density by \cite[Thm.~VII.13.2]{N}.
Consequently, there are infinitely many rational primes $q$, which
split in $K$, such that one prime above $q$ is principal. But if one
prime above $q$ is principal, so is the other.  
\end{proof}

\begin{prop}\label{prop:not}
  Let\/ $d$ be a square-free element of\/ $\ZZ \setminus \{ 0,1 \}$
  and further assume that\/ $d \ne -1$. Let\/ $A$ be a\/ $\ZZ$-linear
  bijection of\/ $\cO$ such that\/ $A(W) = W\!$. Then,
  $A(1) \not\in \sqrt{d\ts} \ts \cO^{\times}\!$.  In particular, this
  conclusion holds for any $A\in \stab (V)$.
\end{prop}

\begin{proof}
  If $A$ is a PM for $V\!$, we have $A(W)=W$ by
  Lemma~\ref{lem:W-to-W}. So, let us more generally assume $A$ to be a
  $\ZZ$-linear bijection of $\cO$ such that $A(W)=W\!$, where we also
  know that $\sqrt{d\ts} \ts \cO^{\times}\! \subseteq W\!$.  Since
  $d\ne -1$ by assumption, $\sqrt{d\ts}$ is not a unit, and
  $\cO^{\times}\nts \ne \sqrt{d\ts} \ts \cO^{\times}$. Now, suppose to
  the contrary of our claim that
  $A(1) \in \sqrt{d\ts}\ts \cO^{\times}\!$, where we may then assume
  $A(1) = \sqrt{d\ts}$ without loss of generality, because multiplying
  $A$ by a unit does not change the type of mapping. Now, we have to
  consider two situations. \smallskip
  
\textit{Case} $1$: $d\equiv 1 \bmod 4$, where we set
  $\delta = \frac{1 + \sqrt{d}}{2}$ and then get
  $A(1) = -1 + 2 \ts \delta$.  Now, let $A(\delta) = a + b \ts \delta$
  with $a,b\in \ZZ$, wherefore $A$ is represented by the matrix
  $\left( \begin{smallmatrix} -1 & a \\ 2 &
      b \end{smallmatrix}\right)$, which must lie in $\GL (2,\ZZ)$ due
  to bijectivity of $A$ on $\cO=\ZZ[\ts\delta\ts ]$.  So, this gives
  $\pm 1 = \det (A) = -2\ts a - b$, and thus
  $A(\delta) = \pm \frac{1}{2} + \frac{b}{2} \ts \sqrt{d\ts}$.  With
  $\sqrt{d\ts} \in W\!$, we also get
  $A(\sqrt{d\ts}\,) = A(2\ts \delta - 1) = \pm 1 + (b-1) \ts
  \sqrt{d\ts} \in W\!$, which implies that its norm,
  $1-d\ts (b-1)^2 \eqdef d^{\ts\prime}$, divides $d^{}_{\nts K} = d$
  by Fact~\ref{fact:square-free-divisor}.
  
  Now, we claim that $d^{\ts\prime}=1$: Suppose to the contrary that
  $d^{\ts\prime}\ne 1$, hence
  $0 \ne 1 - d^{\ts\prime} = d\ts (b-1)^2$. If $d>0$, we have
  $d + d^{\ts\prime} \geqslant 0$ because
  $d^{\ts\prime} \ts\ts | \ts\ts d$, and this gives
  $1+d \geqslant 1-d^{\ts\prime} = d\ts (b-1)^2 \geqslant d$, where
  $d\ts (b-1)^2 \ne 1+d$. Then, $1-d^{\ts\prime} = d$, and $d$ and
  $d^{\ts\prime}$ are coprime.  Since
  $d^{\ts\prime} \ts\ts | \ts\ts d$, this forces $d^{\ts\prime}=\pm 1$
  and thus $d=0$ or $d=2$, which is impossible because
  $d\equiv 1 \bmod 4$. Likewise, if $d<0$, we get
  $d + d^{\ts\prime} \leqslant 0$ and then
  $1+d \leqslant 1 - d^{\ts\prime} =d\ts (b-1)^2 \leqslant d$, which
  is a contradiction. Consequently, we must indeed have
  $d^{\ts\prime}=1$ and hence $b=1$, so $A(\sqrt{d\ts}\,) = \pm 1$,
  where one also gets $a\in\{-1,0\}$.

  Next, choose a prime $q \gg 1$ according to Fact~\ref{fact:density}
  and let $\fp = (\pi) $ be a prime ideal over $q$, where
  $\pi \not\in W$ by construction. Write $\pi = u + v\ts \delta$ with
  $u,v \in \ZZ$, which has norm
\begin{equation}\label{eq:norm-pi}
      N(\pi) \, = \, u^2 + u v +  \myfrac{1-d}{4} \ts v^2
      \, = \, \pm q \ts ,
\end{equation}
where $\frac{1-d}{4}$ is an integer. Now, $q$ cannot divide
$u^2 + u\ts v$, as otherwise $q \gg 1$ implies
$q \ts\ts | \ts\ts v^2$, hence $q\ts\ts | \ts\ts v$ and thus also
$q\ts\ts | \ts\ts u$. But this would give $q^2 \ts\ts | \ts\ts q$,
which is impossible. So, $q$ and $u^2+uv$ are coprime.  Next, we have
$A(\pi) = u \sqrt{d\ts} + v \bigl( \pm \frac{1}{2} + \frac{1}{2} \ts
\sqrt{d\ts}\,)$ with norm
\[
      N(A(\pi)) \, = \, \myfrac{v^2}{4} - d \Bigl( u^2 + u \ts v +
      \myfrac{v^2}{4}\Bigr)  \,  = \,
      \pm \ts q - (d+1) (u^2 + u\ts v)  \ts ,
\]
where \eqref{eq:norm-pi} was used for the second step. Clearly, $q$
can be chosen sufficiently large so that it does not divide
$d+1$. Since $q$ is coprime with $u^2+uv$, we see that
$q \nts\nts \nmid \nts\nts N(A(\pi))$.  Since we know from
Lemma~\ref{lem:coprime} that $A(\pi)$ is coprime with any other
non-ramified prime, we must have $A(\pi) \in W = A(W)$.  As $A$ is
bijective on $W\!$, we get $\pi \in W$ and thus a contradiction, which
rules out $A(1) \in \sqrt{d\ts}\ts \cO^{\times}\!$ in this
case. \smallskip
  
\textit{Case} $2$: $d \equiv 2,3 \bmod 4$, where we consider
$A(1) = \sqrt{d\ts}$ and $A(\sqrt{d\ts}\,) = a + b \ts \sqrt{d\ts}$
with $a,b \in \ZZ$. Here, the determinant condition gives
$\pm 1 = \det \left( \begin{smallmatrix} 0 & a \\ 1 & b
  \end{smallmatrix}\right) = -a$, and thus $a=\pm 1$.
Consequently, $A(\sqrt{d\ts}\,) = \pm 1 + b \ts \sqrt{d\ts} \in
W$. Its norm is given by
$ N(A(\sqrt{d\ts}\,)) = 1 - d \ts b^2 \eqdef d^{\ts\prime}$, which
must divide $d$ (if $d\equiv 2 \bmod 4$) or $2\ts d$ (if
$d \equiv 3 \bmod 4$)  by Fact~\ref{fact:square-free-divisor}.
  
When $d<0$, so $d^{\ts\prime} > 0$, we have $b\in \{ 0, \pm 1 \}$, as
otherwise, due to $d^{\ts\prime}\ts\ts | \ts\ts 2\ts d$, the inequality
$1 + 2\ts d \leqslant 1 - d^{\ts\prime} = d\ts b^2 \leqslant 4\ts d$
gives a contradiction. Likewise, when $d>0$, so $d^{\ts\prime} < 0$,
we again get $b\in \{ 0, \pm 1\}$, as any other value would lead to
$4\ts d \leqslant d \ts b^2 = 1 - d^{\ts\prime} = 1 + \lvert
d^{\ts\prime} \rvert \leqslant 1 + 2 \ts d$ and thus to
$2 \ts d \leqslant 1$, which is impossible. \smallskip
  
\textit{Subcase} $b=0$: Here, we have $A(\sqrt{d\ts}\,) = \pm 1$.  By
Fact~\ref{fact:density}, we may choose a rational prime $q \gg 1$ such
that $(q) = \fp \ts \bar{\fp}$ with $\fp $ principal, so $\fp = (\pi)$
for some $\pi = u + v\ts \sqrt{d\ts}$. Then, we get
\[
  \pm \ts q \, = \, N(\pi) \, = \, u^2 - d\ts v^2
  \, = \,  u^2 - v^2 - (d-1) v^2 .
\]
Now, we must have $q \nts\nmid \nts (u^2 - v^2)$: Otherwise, $q \gg 1$
forces $q\ts\ts | \ts\ts v$, which implies $q \ts\ts | \ts\ts u$ and
then $q^2 \ts\ts | \ts\ts q$, and thus a contradiction. Next, we
calculate $A(\pi) = u \ts \sqrt{d\ts} \pm v$ and thus
\[
  N(A(\pi)) \, = \, v^2 - d \ts u^2 \, = \,
  \pm \ts q + (d+1) (v^2 - u^2) \ts .
\]
Since $q$ does not divide either of the two bracketed terms, we see
that $q \nts \nmid \nts N(A(\pi))$. As before, we conclude that
$A(\pi) \in W = A(W)$ and hence $\pi \in W$ by the bijectivity of $A$.
This contradicts the original choice of $\pi$, with the same
conclusion as in Case $1$. \smallskip

\textit{Subcase} $b=\pm 1$: Here, we have $d = 1-d^{\ts\prime}$, so
$d$ and $d^{\ts\prime}$ are coprime integers. Then,
$d^{\ts\prime} \ts\ts | \ts\ts 2 \ts d$ with $d>0$ means
$d^{\ts\prime} \in \{ -1, -2\}$, and thus $d=2$ or $d=3$, which are
two of the cases from Fact~\ref{fact:real-ex}.  In its proof,
  $A(W)=W$ was used to show $A(V) \subseteq V$, and our claim holds.
Likewise, when $d<0$, we can only have $d^{\ts\prime} \in \{ 1,2 \}$,
and thus $d=0$, which is excluded, or $d=-1$, which is the excluded
case of the Gaussian integers  (where the claim actually does not
  hold).
\end{proof}

At this point, recalling Fact~\ref{fact:Gauss} for $d=-1$, we can
completely answer the question for the stabiliser of $V$ in imaginary
quadratic fields as follows.

\begin{theorem}\label{thm:imag}
  Let\/ $K = \QQ(\sqrt{d\ts}\,)$, with\/ $d<0$ square-free, be an
  imaginary quadratic field, with ring of integers\/
  $\cO = \cO^{}_{\! K}$. Then, any\/ $A \in \stab(V)$ is of the form\/
  $A = m^{}_{\varepsilon}\nts \circ \sigma$ for some\/
  $\varepsilon \in \cO^{\times}\!$ and\/
  $\sigma \in \Gal (K/\QQ) \simeq C_2$. Every such mapping is
  bijective on\/ $V\!$, and we obtain that the stabiliser of\/ $V$ is
  a group,
\[  
  \stab (V) \, = \, \cO^{\times}\! \rtimes
  \Gal (K/\QQ) \, \simeq \, C_n \rtimes C_2
  \, = \, D_n \ts ,
\]
where\/ $D_n$ is the dihedral group, here with\/ $n=4$ for\/ $d=-1$,
$n=6$ for\/ $d=-3$, and\/ $n=2$ in all remaining cases.
\end{theorem}

\begin{proof}
  The cases $d\in \{ -1,-2,-3\}$ are known from
  Proposition~\ref{prop:min-two} and Facts~\ref{fact:Gauss} and
  \ref{fact:Eisen}, so we may restrict to $d \leqslant -5$, as $-4$ is
  not square-free.  Let $A$ be a PM for $V\!$, so we know that
  $A(W) = W$ from Lemma~\ref{lem:W-to-W}. In view of
  Proposition~\ref{prop:unit-suffices}, we now need to show
  $A(1) \in \cO^{\times}\!$.

  Suppose to the contrary that $A(1) \not\in \cO^{\times}\!$. Then,
  there exists a ramified prime $p$ such that $A(1) \in \fp$ where
  $\fp$ is the prime ideal over $p$, so $(p) = \fp^2$.  As $A(1)$ is
  square{\ts}-free, we know that $N(A(1))$ is exactly divisible by
  $p$, hence \emph{not} by $p^2$. \smallskip

  \textit{Case} $1$: $d \equiv 2,3 \bmod 4$, with
  $\cO = \ZZ [ \sqrt{d\ts}\,]$ and $d^{}_{\nts K} = 4 d$.  Let
  $A(1) = a + b \ts \sqrt{d\ts}$ with $a,b \in \ZZ$, so
  $N(A(1)) = a^2 - d \ts b^2 \eqdef d^{\ts\prime} \ts\ts | \ts\ts 2 d$
  because $A(1) \in W$.

  If we had $b= 0$, we would get
  $d^{\ts\prime} = a^2 \ts\ts | \ts\ts 2 d$.  But $d$ is
  square{\ts}-free and $A(1) \not\in \cO^{\times}\!$ by assumption,
  which implies $d^{\ts\prime} = 4$ and $a=\pm 2$.  Consequently,
  $A(1) = \pm 2 \in \fp^2_2$, which is not square{\ts}-free because
  $2$ is ramified. So, we see that $b\ne 0$, and we actually must have
  $b=\pm 1$ (otherwise,
  $d^{\ts\prime} = a^2 - d \ts b^2 \geqslant - d \ts b^2 \geqslant 4
  \ts\ts \lvert d \rvert > 2 \ts\ts \lvert d \rvert \geqslant
  d^{\ts\prime}$ would give a contradiction). So, with $b=\pm 1$, we
  get $d^{\ts\prime} = a^2 - d > 0$.

  Now, we see that $a\ne 0$, as otherwise $A(1) = \pm \sqrt{d\ts}$,
  which contradicts Proposition~\ref{prop:not}. Also, $d^{\ts\prime}$
  must be even (otherwise, $d^{\ts\prime} \ts\ts | \ts\ts d$ and
  $d^{\ts\prime} = a^2 - d > -d \geqslant d^{\ts\prime}$ gives a
  contradiction). So, write $d^{\ts\prime} = 2 \tilde{d}$, where
  $\tilde{d} \ts\ts | \ts\ts d$ with $\tilde{d} > 0$.  Since $d$ is
  square{\ts}-free, $\tilde{d}$ must divide $a$, and we get
\[
   2 \ts \tilde{d} \, = \, a^2 - \ts d \, \geqslant \,
   \tilde{d}^{\ts\ts 2} \nts + \ts \tilde{d} \, = \, 
   \bigl( \tilde{d} + 1 \bigr)  \tilde{d} \ts ,
\]
which implies $\tilde{d}=1$. But this means $2 = a^2 - d \geqslant 5$,
which is absurd, so $A(1) \not\in \cO^{\times}\!$ is ruled out in this
case. \smallskip

\textit{Case} $2$: $d\equiv 1 \bmod 4$, where $\cO = \ZZ[\delta]$ with
$\delta = \frac{1 + \sqrt{d\ts}}{2}$ and $d^{}_{\nts K} = d$. Let
$A(1) = a + b \ts \delta$ with $a,b\in \ZZ$.  Here, we get
$N(A(1)) = \bigl( a + \frac{b}{2}\bigr)^{\! 2} - d \bigl(
\frac{b}{2}\bigr)^2 \eqdef d^{\ts\prime}$, which divides $d$.

As before, $b=0$ is impossible, as it would imply
$a^2 \ts\ts | \ts\ts d$ with $d$ square{\ts}-free, so $a=\pm 1$ and
thus $A(1) \in \cO^{\times}\!$ in contradiction to our assumption. We
claim that, once again, we must have $b=\pm 1$: Otherwise, we would
have
$d^{\ts\prime} \geqslant - d \bigl( \frac{b}{2}\bigr)^{\! 2} \geqslant
- d$, hence $d^{\ts\prime} = -d$ and $b=\pm 2$ together with
$a=-\frac{b}{2}$. This, in turn, would give $A(1) = \pm \sqrt{d\ts}$,
in contradiction to Proposition~\ref{prop:not}.
  
So, we have
$0 \leqslant \bigl( a \pm \frac{1}{2}\bigr)^{\! 2} = d^{\ts\prime} +
\frac{d}{4}$, which implies
$\lvert d \rvert \geqslant d^{\ts\prime} \geqslant \frac{\lvert d
  \rvert}{4}$, hence $d^{\ts\prime} = \lvert d \rvert$ or
$d^{\ts\prime} = \frac{ \lvert d \rvert }{3}$ because $d$ is odd. In
the first case, we get
\[
  \Bigl( a \pm \myfrac{1}{2} \Bigr)^{\! 2} \, = \,
  \lvert d \rvert - \myfrac{\lvert d \rvert}{4} \, = \,
  \myfrac{3 \ts\ts \lvert d \rvert}{4} \ts ,
\]
which forces $3 \ts\ts \lvert d \rvert$ to be a square in $\ZZ$.  With
$d$ being square{\ts}-free, this is only possible for $d=-3$, which
was excluded. When $d^{\ts\prime} = \frac{\lvert d \rvert}{3}$, we
obtain
\[
  \Bigl( a \pm \myfrac{1}{2} \Bigr)^{\! 2} \, = \,
  \myfrac{\lvert d \rvert}{3} -
  \myfrac{\lvert d \rvert}{4} \, = \,
  \myfrac{\lvert d \rvert}{12} \ts ,
\]
and $\frac{\lvert d \rvert}{3}$ is a square in $\ZZ$, which again
leads to the excluded case $d=-3$. So, $A(1) \not\in \cO^{\times}\!$
is ruled out also in this case, and we have
$A(1) \in \cO^{\times}\!$.\smallskip

As each mapping of the form $A = m^{}_{\varepsilon}\nts \circ \sigma$
is in the stabiliser, the structure of $\stab (V)$ derives from the
unit group (as stated early in this chapter) together with
$\Gal (K/\QQ) \simeq C_2$ and the relation
$\sigma \circ m^{}_{\varepsilon} = m^{}_{\sigma(\varepsilon)} \nts
\circ \sigma$.
\end{proof}

Let us next attack the more complicated case of real quadratic fields,
where we begin with an observation that follows by elementary
arguments from our previous results.

\begin{prop}\label{prop:powers}
  Let\/ $d = p > 0$ be a rational prime, and consider\/
  $K = \QQ(\sqrt{d\ts}\,)$. Then, any $A \in\stab(V)$ is of the form\/
  $A = m^{}_{\varepsilon}\nts \circ \sigma$ with\/
  $\varepsilon \in \cO^{\times}\!$ and\/ $\sigma \in \Gal (K/\QQ)$.
\end{prop}

\begin{proof}
  The case $d=2$, where $d^{}_{\nts K} = 8$, is covered by
  Fact~\ref{fact:real-ex}.  Next, consider the case
  $d^{}_{\nts K} \nts = d = p \equiv 1 \bmod 4$, where we know that
  $W = \cO^{\times}\ts\ts \dot{\cup}\, \sqrt{p} \, \cO^{\times}\!$ is
  a disjoint union. From Lemma~\ref{lem:W-to-W}, we get
  $A(1) \in W\!$, while Proposition~\ref{prop:not} asserts that
  $A(1) \not\in \sqrt{p}\, \cO^{\times}\!$, so
  $A(1) \in \cO^{\times}\!$. Then,
  Proposition~\ref{prop:unit-suffices} shows that $A$ is of the form
  claimed.

  Next, let $p\equiv 3 \bmod 4$, where $(2) = \fp^2_2$ is
  ramified. Consequently, we have
\[
     W \, = \, \cO^{\times} \ts \ts \dot{\cup} \, \sqrt{p}\, \cO^{\times}
     \ts\ts \dot{\cup} \, W'
\]
for some subset $W' \subseteq \fp^{}_{2}$. In view of
Proposition~\ref{prop:unit-suffices}, we now have to show that
$A(1) \in \cO^{\times}\!$.

Here, we know that $A(1) \in W\!$, but
$A(1) \not\in \sqrt{p}\, \cO^{\times}\!$ by
Proposition~\ref{prop:not}. Suppose we had \mbox{$A(1)\in W' \!$},
hence $A(1)\in \fp^{}_{2}$. Since $A(\cO) = \cO$, we must have
$A(\sqrt{p}\, ) \not\in \fp^{}_{2}$, while we still have the inclusion
$A(\sqrt{p}\, ) \in W$ by Lemma~\ref{lem:W-to-W}.  If
$A (\sqrt{p}\, ) = \epsilon \in \cO^{\times}\!$, we get
$\bigl( A^{-1} \nts\nts \circ m^{}_{\epsilon}\bigr) (1) = \sqrt{p}$,
in contradiction to Proposition~\ref{prop:not} applied to the mapping
$A^{-1} \nts \circ m^{}_{\epsilon}$, which clearly is a $\ZZ$-linear
bijection of $\cO$ that maps $W$ onto itself. Consequently, we have
$A(\sqrt{p}\, ) \in \sqrt{p} \, \cO^{\times}\!$.

Possibly after multiplying by a unit, which means no loss of
generality, we may assume $A(\sqrt{p}\, ) = \nts \sqrt{p}\,$.  Now,
let $A(1) = a + b \ts \sqrt{p}$ with $a,b\in\ZZ$. Since $A(1) \in W'$,
its norm must satisfy $N(A(1)) \in \{ \pm 2, \pm 2 \ts p
\}$. Moreover, we have
$ a = \det \left( \begin{smallmatrix} a & 0 \\ b &
    1 \end{smallmatrix}\right) = \det(A) = \pm 1$, so we get
$N(A(1)) = 1 - p \ts b^2$.  This never equals $2$ or $\pm 2 p$, while
it can agree with $- 2$ only for $p=3$. Since $p=3$ is known from
Fact~\ref{fact:real-ex}, we get $A(1) \in \cO^{\times}$ in all cases,
and we are done.
\end{proof}

The simplest cases not yet covered are $\QQ (\sqrt{6}\,)$ and
$\QQ (\sqrt{15}\,)$, which can be treated by explicit arguments
similar to those used for our previous cases from
Fact~\ref{fact:real-ex}, taking into account that one now has two
ramified primes (as in the case of $\ZZ[\sqrt{3}\,]$ treated in
\cite{BBHLN}). The result is the expected one, and completely in line
with the above. To generalise this now to all real quadratic fields,
we invoke another result on quadratic forms of a more geometric
origin. If $Q$ is  a given quadratic form, we (uniquely)
represent it by the corresponding symmetric matrix $B_Q$, and call
$\det (B_Q)$ the \emph{determinant} of $Q$.

\begin{lemma}\label{lem:quadratic-form}
  Let\/ $Q^{}_1$ and $Q^{}_2$ be two binary quadratic forms
  over\/ $\RR^2$ with negative determinant, and assume that they share
  a non-empty level curve, which is to say that, for some\/
  $c^{}_{1}, c^{}_{2} \in \RR$, the non-empty curves\/
  $\{ Q^{}_1 (x,y) = c^{}_{1} \}$ and\/ $\{ Q^{}_2 (x,y) = c^{}_{2}\}$
  agree. Then, there exists a constant\/ $0\ne c^* \in\RR$ such that\/
  $Q^{}_2 = c^* \, Q^{}_1$.
\end{lemma}

\begin{remark}
  Note that the assertion of Lemma~\ref{lem:quadratic-form} holds for
  any pair of conics, not necessarily hyperbolic, under the assumption
  that the common level curve is infinite.  \exend
\end{remark}

\begin{proof}
  Due to the determinant condition, the curves $\{ Q_i (x,y) = c \}$
  are hyperbolas, with the limiting (degenerate) case $c=0$ consisting
  of two straight lines each. If $\{ Q^{}_{1} (x,y) = c^{}_{1}\}$ and
  $\{ Q^{}_{2} (x,y) = c^{}_{2} \}$ agree as curves, they must have
  the same asymptotes, which are two straight lines passing through
  the origin. They are determined by the equation $Q_{i} (x,y) = 0$,
  for either choice of $i$, where we have a factorisation
\[
     Q^{}_i (x,y) \, = \, (a^{}_{i,1} x + b^{}_{i,1} y) 
        (a^{}_{i,2} x + b^{}_{i,2} y)
\]  
  for suitable real numbers $a^{}_{i,j}$ and $b^{}_{i,j}$ with
  $i,j \in \{ 1,2\}$.
  
  Now, possibly after interchanging the two factors for one of the
  forms, we may assume that $a^{}_{i,1} x + b^{}_{i,1} y =0$, for both
  $i$, is the equation for the first asymptote, which means that
  $a^{}_{i,2} x + b^{}_{i,2} y =0$ determines the other one. Observe
  that two non-degenerate linear equations determine the same line if
  and only if one equation is a constant (and non-zero) multiple of
  the other. In our case at hand, this means that there are constants
  $\kappa^{}_j \ne 0$ such that
  $a^{}_{2,j} x + b^{}_{2,j} y = \kappa^{}_{j} (a^{}_{1,j} x +
  b^{}_{1,j} y)$, hence
\[
    Q^{}_{2} (x,y) \, = \, \kappa^{}_{1} \ts \kappa^{}_{2} 
    \, Q^{}_{1} (x,y) \ts ,
\]  
which gives our claim by taking $c^* = \kappa^{}_{1} \ts \kappa^{}_{2}$.
\end{proof}

In the context of real quadratic fields,
$Q(a,b) = N( a + b \ts \delta)$ with $a,b \in \ZZ$ defines a quadratic
form over $\ZZ^2$, where $\delta$ is chosen as in
Eq.~\eqref{eq:basis}.  Clearly, $Q$ extends to a quadratic form over
$\RR^2$, and we have $\det (B_Q) = -\frac{1}{4} d^{}_{K} < 0$ in all
cases under consideration. This can now be used as follows.

\begin{prop}\label{prop:geo-2}
  Let\/ $K$ be a real quadratic field with\/ $\cO = \cO^{}_{\nts K}$
  as its ring of integers, and let\/ $W$ be the set of square-free
  elements of\/ $\cO$ that are coprime with all non-ramified rational
  primes. If $A$ is a\/ $\ZZ$-linear bijection of\/ $\cO$ with\/
  $A(W) = W\!$, one also has\/ $A(\cO^{\times}) = \cO^{\times}\!$.
\end{prop}

\begin{proof}
  As in the proof of Lemma~\ref{lem:W-to-W}, we can write $W$ as a
  finite disjoint union of level sets
   $S'_c = \{ a + b \ts \delta : a,b \in \ZZ \text{ and } Q (a,b) = c
  \}$ with the $\delta$ from \eqref{eq:basis}, where $c$ divides
  $2 \ts d$.  We identify $S_c'$ with a subset $S_c$ of $\ZZ^2$
  as in \eqref{eqn:S_c}. Then each
  $S_c$ is the intersection of the curve
  $\{ (x,y) \in \RR^2 : Q(x,y) = c\}$ with $\ZZ^2$, where we are in
  the situation of Lemma~\ref{lem:quadratic-form}.  In particular, the
  non-trivial level curves are hyperbolas.

  As before, we use $A$ both for the given mapping and for its
  $\GL(2,\ZZ)$-representation after the identification of $\cO$ with
  $\ZZ^2$ via the $\ZZ$-basis $\{ 1, \delta \ts \}$ of $\cO$. Since
  $A(W)=W$, we get
\[
  A (S^{}_{1}) \, =  \bigcup_{c \ts |  2 d}
  S_c \cap A(S^{}_{1}) \ts ,
\]
which is a finite union of disjoint sets. Consequently, for some
$c \ts\ts | \ts\ts 2 \ts d$, the set $S_c \cap A(S^{}_{1})$ must be
infinite. This implies that the hyperbolas $\{ Q(x,y) = c\}$ and
$\big\{ \bigl(Q\circ A^{-1}\bigr) (x,y) = 1 \big\}$ match in more than
five points and thus have to agree as curves, where the latter
corresponds to the image of $\{ Q(x,y) = 1\}$ under $A$. By
Lemma~\ref{lem:quadratic-form}, there must be a real number
$c^* \ne 0$ such that $Q\circ A^{-1} = c^* Q$.
  
Now, again inspecting the proof of Lemma~\ref{lem:W-to-W}, we know
that some power of $A$ maps units to units. In fact, from the above
argument, we know that $A^k (S^{}_{1}) = S^{}_{1}$ must hold for some
$k\in\NN$. Then, $A^{-k}$ maps $S^{}_{1}$ into itself and
$A^{-k} (1)$ has norm $1$. This implies
\[
     1 \, = \, Q \bigl( A^{-k} (1) \bigr) \, = \,
     (c^*)^{k}\ts Q(1) \, = \, (c^*)^{k} ,
\]
which means that $c^*$ is a root of unity, as $Q$ is
non-degenerate. But $c^* \in \RR$ by construction, so
$c^* \in \{ \pm 1 \}$, which implies that $A$ maps elements of norm
$\pm 1$ to elements of norm $\pm 1$, that is, units to units.
\end{proof}

We can now wrap up this part as follows.

\begin{theorem}\label{thm:real}
  Let\/ $K = \QQ (\sqrt{d\ts}\,)$, with\/ $d>1$ square-free, be a real
  quadratic field, with ring of integers\/ $\cO = \cO^{}_{\nts K}$.
  Then, any $A\in \stab (V)$ is of the form\/
  $A = m^{}_{\varepsilon}\nts \circ \sigma$ with\/
  $\varepsilon \in \cO^{\times}\!$ and\/
  $\sigma \in \Gal (K/\QQ) \simeq C_2$. Every such mapping is
  bijective on\/ $V\!$, and the stabiliser is a group,
\[
  \stab (V) \, = \, \cO^{\times} \! \rtimes \Gal (K/\QQ)
  \, \simeq \, (C_2 \times C_{\infty}) \rtimes C_2
  \, \simeq \, C_2 \times D_{\infty} \ts ,
\]
where\/ $D_{\infty} = C_{\infty}\nts\rtimes C_2$ is the infinite
dihedral group.
\end{theorem}

\begin{proof}
  If $A$ is a PM for $V\!$, we once again know $A(W)=W$ from
  Lemma~\ref{lem:W-to-W}. Then, $A$ maps units to units by
  Proposition~\ref{prop:geo-2}, and $A$ is of the claimed form as a
  result of Proposition~\ref{prop:unit-suffices}, hence clearly a
  bijection on $V\!$.

The group property of $\stab (V)$ is then obvious, and its calculation
follows from the standard properties mentioned earlier, including
the structure of the unit group, $\cO^{\times}\!$.
\end{proof}

Given any quadratic field, we know from Lemma~\ref{lem:2-suffices}
that $\stab (V^{}_{k}) \subseteq \stab (V^{}_{2})$ holds for all
$k\geqslant 2$.  Now, $\stab (V^{}_{2})$ is a group, all elements of
which also preserve $V^{}_{k}$,  by Theorems \ref{thm:imag} and
  \ref{thm:real},
so
$\stab (V^{}_{k}) = \stab (V^{}_{2})$, and we obtain the following
conclusion.

\begin{coro}\label{coro:stab}
  Let\/ $K$ be a quadratic field, imaginary or real, with ring of
  integers\/ $\cO$. Then, for any\/ $k\in\NN$ with\/ $k\geqslant 2$,
  the set\/ $\stab (V^{}_{k})$ is a group that is independent of\/
  $k$, namely the group\/ $\stab (V)$ characterised by
  Theorems~\textnormal{\ref{thm:imag}} and
  \textnormal{\ref{thm:real}}.  \qed
\end{coro}

\section{The subshift of \textit{k}-free integers}\label{sec:shift}

The connection to dynamical systems emerges from the observation that
$V \nts = V^{}_{\nts k}$ defines an element of
$\XX^{}_{\cO} \defeq \{ 0,1 \}^{\cO}$, which is compact in the product
topology, by identifying $V$ with the function (or configuration)
$1^{}_{V}$. As usual, we write elements of $\XX^{}_{\cO}$ as
$u = (u^{}_{x})^{}_{x\in\cO}$. Now, one can define an $\cO$-action
$\alpha \! : \, \cO \times \XX^{}_{\cO} \xrightarrow{\quad}
\XX^{}_{\cO}$ via $(t,u)\mapsto \alpha^{}_{t} (u)$ with
$(\alpha^{}_{t} (u))^{}_{x} \defeq u^{}_{x+t}$. This action is
continuous and turns $(\XX^{}_{\cO}, \cO)$ into a \emph{topological
  dynamical system} (TDS).

For a function $f \! : \, \XX^{}_{\cO} \longrightarrow \CC$, we define
the translation action via
$\bigl( \alpha^{}_{t} f \bigr) (u) = f (\alpha^{}_{-t} u )$. With
this, we get $\alpha^{}_{t} ( 1^{}_{V} ) = 1^{}_{t+V}$, where
$1^{}_{S}$ denotes the characteristic function of a set
$S \subseteq \ZZ^2$. Now, the orbit closure in the product topology,
\[
  \XX^{}_{V} \, \defeq \, \overline{ \{
    \alpha^{}_{t} (1^{}_{V}) : t \in \cO \} } \ts ,
\]
is a closed and hence compact subset of $\XX^{}_{\cO}$ that is
invariant under the above $\cO$-action, so $(\XX^{}_{V}, \cO)$ is a
TDS as well. At this point, it is useful to employ the Minkowski
embedding of $\cO$ as a lattice $\vG\subset \RR^2$, for which there
are several possibilities.

To establish the link with symbolic dynamics, it is more convenient to
 employ the standard $\ZZ$-basis $\{ 1, \delta \}$ of $\cO$
  with the $\delta$ from \eqref{eq:basis}
and then consider
\begin{equation}\label{eq:def-Vprime}
  V^{\prime}_{\nts k} \, \defeq \, \big\{
  (m^{}_{1}, m^{}_{2}) \in \ZZ^2 :
  m^{}_{1} + m^{}_{2} \delta \in V^{}_{\nts k} \big\} ,
\end{equation}
which is now a subset of $\ZZ^2$. This also identifies $\cO$ with
$\vG = \ZZ^2$ in a specific way. We are now working with binary
subshifts of $\{ 0,1 \}^{\ZZ^2}\!$, hence in the usual setting of
symbolic dynamics \cite{LM,Klaus}. In particular, for fixed $k$, we
now have
$\XX = \XX^{}_{V'_{k}} = \overline{ \{ t + V'_{k} : t \in \ZZ^2 \} }$,
where we tacitly identify subsets of $\ZZ^2$ with their characteristic
functions. This gives us the TDS $(\XX, \ZZ^2)$, which we call the
subshift induced by the $k$-free integers of $\cO$. It is unique up to
isomorphism.

Let $\iota \! : \, \cO \xrightarrow{\quad} \ZZ^2$ be the embedding
defined as above by $1\mapsto (1,0)$ and $\delta \mapsto (0,1)$. When
$\fb$ is an ideal in $\cO$, its embedding
$\vG_{\nts\fb} \defeq \iota (\fb)$ is a sublattice of $\ZZ^2$ of index
\[
    [\ZZ^2 : \vG_{\nts\fb}] \, = \, \No (\fb) \ts .
\]
Now, $\XX$ can also be seen as an algebraic $\cB$-free lattice system 
in the sense of \cite[Def.~5.1]{BBHLN}, where we assume
$k\geqslant 2$ to be fixed. Let $\fP$ denote the set of
prime ideals of $\cO$ and consider
\begin{equation}\label{eq:BB-def}
    \cB \, \defeq \, \{ \vG_{\nts\fp^k}: \fp \in \fP \}
    \quad \text{with }  \vG_{\nts\fp^k} = \iota (\fp^k) \subset \ZZ^2 .
\end{equation}
Here, $\cB$ is an infinite set of coprime sublattices of $\ZZ^2$,
which is to say that $\vG_{\nts\fp^k} + \vG_{\nts\fq^k} = \ZZ^2$
whenever the prime ideals $\fp$ and $\fq$ are different. The defining
set from \eqref{eq:def-Vprime} can now be rewritten as
\begin{equation}\label{eq:def-VB}
  V'_{k} \, = \, V^{}_{\cB} \, \defeq \,
  \ZZ^2 \setminus \bigcup_{\fp\in\fP} \vG_{\nts\fp^k} \ts .
\end{equation}
This gives another way to view our subshift as an orbit closure,
namely $\XX = \XX^{}_{\cB} = \overline{\ZZ^2 + V^{}_{\cB}}$, 
  where the latter is a shorthand for the translation orbit closure of
  $V^{}_{\cB}$ under the shift action.  In our special setting, where
$k\geqslant 2$, we also get
\[
  \sum_{\fp\in\fP} \myfrac{1}{[\ZZ^2 : \vG_{\nts\fp}]} \, =
  \sum_{\fp\in\fP} \myfrac{1}{\No (\fp)^k} \, < \, \infty \ts ,
\]
which is to say that the $\cB$-free system $(\XX^{}_{\cB}, \ZZ^2)$ 
is automatically Erd\H{o}s; compare \cite{BBHLN}.

\begin{lemma}\label{lem:density}
  Let\/ $K$ be a quadratic field, with ring of integers\/
  $\cO=\cO^{}_{\! K} = \ZZ[\delta]$, and let\/ $V'_{k}\subset \ZZ^2$
  with\/ $k\geqslant 2$ be the set defined in \eqref{eq:def-Vprime}.
  Then, the set\/ $V'_{k}$ has tied density\/
  $1/\zeta^{}_{\nts K} (k)$, where\/ $\zeta^{}_{\nts K}$ denotes the
  Dedekind zeta function of\/ $K$.
\end{lemma}

\begin{proof}
  Tied density means that one considers
  $\card \bigl( V'_{k} \cap B_r (0)\bigr)/ \pi r^2$ in the limit
  $r \to \infty$, or with disks replaced by centred squares and their
  areas, which is known to exist and to be independent of the
  averaging sequence, once it is centred and of van Hove type. In
  \cite{BMP}, the arguments are spelled out in detail for the visible
  lattice points of $\ZZ^2$, and the same approach works here as well.
  
  Indeed, looking at \eqref{eq:def-VB}, it is clear that $V'_k$ is a
  point set that is the limit, in the local topology, of a nested
  sequence of lattices by considering
  $\ZZ^2 \setminus \bigcup_{\fp: \No (\fp) \leqslant n}
  \vG_{\nts\fp^k}$ as $n\to\infty$. Clearly, the density of each
  lattice in this sequence exists with respect to the averaging
  sequence and, via a standard inclusion-exclusion argument, is given
  by
\[
     \prod_{\fp : \No (\fp) \leqslant n} \bigl( 1 - \No (\fp)^{-k} \bigr) ,
\]  
  which decreases in $n$ and converges to $1/\zeta^{}_{\nts K} (k)$ 
  as claimed.
\end{proof}

Also, the set $V'_{k}$ is a weak model set of maximal density in the
sense of \cite{BHS}. Indeed, with $k$ fixed, we can set
$H_{\fp^k} = \ZZ^2 / \vG_{\nts\fp^k}$ for each prime ideal $\fp$ in
$\cO$, which defines an Abelian group of order $\No (\fp)^k$. Then,
$H \defeq \prod_{\fp\in\fP} H_{\fp^k}$ is a compact Abelian group,
which serves as internal space for the
 \emph{cut and project scheme} (CPS) 
\begin{equation}\label{eq:CPS}
\renewcommand{\arraystretch}{1.2}\begin{array}{r@{}ccccc@{}l}
   \\  & \RR^2 & \xleftarrow{\;\;\; \pi \;\;\; } 
   & \RR^2 \nts\nts \times \nts\nts H & 
   \xrightarrow{\;\: \pi^{}_{\text{int}} \;\: } & H & \\
   & \cup & & \cup & & \cup & \hspace*{-1.5ex} 
   \raisebox{1pt}{\text{\footnotesize dense}} \\
   & \ZZ^2 & \xleftarrow{\;\ts 1-1 \;\ts } &
     \cL  & \xrightarrow{ \qquad } &\pi^{}_{\text{int}} (\cL) & \\
   & \| & & & & \| & \\
   & L & \multicolumn{3}{c}{\xrightarrow{\qquad\quad\quad
    \,\,\,\star\!\! \qquad\quad\qquad}} 
   &  {L_{}}^{\star\nts}  & \\ \\
\end{array}\renewcommand{\arraystretch}{1}
\end{equation}
where $\cL$ is the standard diagonal embedding of $\ZZ^2$
into $\RR^2 \times H$;  see \cite{TAO} for background.

Now, defining the $\star$-image of $x\in\ZZ^2$ as the lift of $x$ into
$H$ by using its value modulo $\vG_{\nts\fp^k}$ at place $\fp^k$, one
obtains
\[
  V'_{k} \, = \, \bigl\{ x \in \ZZ^2 : x \bmod \vG^{}_{\!\fp^k}
  \ne 0 \text{ for all } \fp \in \fP \bigr\} .
\]
In other words, the subset
$W = \{ (h_{\fp^k})^{}_{\fp \in \fP} : h_{\fp^k} \ne 0 \text{ for all
} \fp\in\fP \}$ provides a coding for $V'_{k}$, in the sense that $W$
  is the window for the description of $V'_{k}$ as a model set in the
  CPS \eqref{eq:CPS}. In particular, $V'_{k}$ and many of its
  properties can be retrieved from $W$ via the CPS.  In the natural
(and normalised) Haar measure of $H$, this set has volume
\[
  \vol (W) \, = \,
  \prod_{\fp\in\fP} \frac{\No (\fp)^k - 1}{\No (\fp)^k}
  \, = \, \prod_{\fp\in\fP} \bigl( 1 - \No (\fp)^{-k}\bigr)
  \, = \, \frac{1}{\zeta^{}_{\nts K} (k)} \ts ,
\]
where $\zeta^{}_{\nts K}$ denotes the Dedekind zeta function of our
quadratic field under consideration. Note that this is nothing but a
variant of the argument used in the proof of Lemma~\ref{lem:density}.

In this setting, we get the reformulation of $V'_{k}$ as
\[
    V'_{k} \, = \, \{ x \in \ZZ^2 : x^{\star} \in W \ts \} \ts ,
\]
which means that we have recognised it as a \emph{weak model set}. The
term `weak' here emerges from the observation that the set $W$, which
is also known as the \emph{window}, is a compact subset of $H$ that
has no interior. It thus consists of boundary only, and the latter has
Haar measure $1/\zeta^{}_{\nts K} (k)$ as derived above.  Now, by the
density formula for weak model sets \cite[Prop.~3.4]{HR}, the density
of our set is bounded from above by
$\dens (\ZZ^2) \ts \vol (W) = \vol (W)$, which agrees with the density
of $V'_{k}$ by Lemma~\ref{lem:density}.

Next, we call a set $U\subset \ZZ^2$ \emph{admissible} for $\cB$ from
\eqref{eq:BB-def} if, for every $\vG_{\nts\fp} \in \cB$, the set $U$
meets at most $\No (\fp)^k - 1$ cosets of $\vG_{\nts\fp}$ in $\ZZ^2$,
that is, misses at least one. The collection of all admissible subsets
of $\ZZ^2$ constitutes again a subshift, denoted by $\AAA$, which
clearly contains $\XX^{}_{\cB}$ by construction.

\begin{prop}\label{prop:admissible}
  Let\/ $K$ be a quadratic field, with ring of integers\/
  $\cO = \cO^{}_{\! K}$ and\/ $k$-free elements\/ $V^{}_{k}$ for some
  fixed\/ $k \geqslant 2$. If\/ $\XX^{}_{\cB}$ is the\/ $\cB$-free
  shift induced by\/ $V^{}_{k}$, and\/ $\AAA$ the corresponding shift
  of admissible sets, one has\/ $\XX^{}_{\cB} = \AAA$. In particular,
  $\XX^{}_{\cB}$ is hereditary: Arbitrary subsets of elements of\/
  $\XX^{}_{\cB}$ are again elements of\/ $\XX^{}_{\cB}$.
\end{prop} 

\begin{proof}
  While the relation $\XX^{}_{\cB} \subseteq \AAA$ is clear by
  construction, the converse is the non-trivial part of the statement.
  If follows via \cite[Prop.~5.2]{BBHLN}, which rests on an asymptotic
  density argument that holds because our system is Erd\H{o}s.  
    As such, the claim is a special case of \cite[Thm.~5.3]{BBHLN}.
   
   Since subsets of admissible sets clearly remain admissible,
   $\XX^{}_{\cB}$ is hereditary. 
\end{proof}

The action of $\cG\defeq \ZZ^2$ on $\XX = \XX^{}_{V'_k}$ is faithful. 
 We now  consider the groups
\[
\begin{split}    
  \cS (\XX) \, & \defeq \, \cent^{}_{\Aut (\XX)} (\cG) \, = \,
  \{ H \in \Aut (\XX) : GH = HG \text{ for all } G \in \cG \}
    \quad\text{and} \\
  \cR (\XX) \, & \defeq \, \norm^{}_{\Aut (\XX)} (\cG) \, = \,
  \{ H\in\Aut (\XX) : H\cG H^{-1} = \cG \}
  \ts ,
\end{split} 
\] 
where $\Aut (\XX)$ refers to the group of all homeomorphisms
of $\XX$. These two groups are also known as the (topological)
\emph{centraliser} and \emph{normaliser}, respectively.

\begin{prop}\label{prop:trivial}
  Let\/ $k\geqslant 2$ be fixed, and let\/ $(\XX^{}_{\cB}, \ZZ^2)$ be
  the\/ $\cB$-free system from
  Proposition~\textnormal{\ref{prop:admissible}}. Then, the
  centraliser is the trivial one, $\cS (\XX^{}_{\cB}) = \cG$, and the
  normaliser is of the form\/
  $\cR (\XX^{}_{\cB}) = \cS (\XX^{}_{\cB}) \rtimes \cH$, where\/ $\cH$
  is isomorphic with a non-trivial subgroup of\/ $\GL(2,\ZZ)$.
\end{prop}

\begin{proof}
  This is a consequence of \cite[Thm.~5.3]{BBHLN}. In fact, the
  triviality (or minimality) of the centraliser employs an argument
  put forward by Mentzen \cite{Mentzen} for the subshift of
  square-free integers, which was then extended to lattice systems in
  \cite{BBHLN}.
  
  With this structure of the centraliser, a variant of the
  \emph{Curtis--Hedlund--Lyndon} (CHL) theorem, compare \cite{LM}, can
  be used to prove that any element of $\cR (\XX^{}_{\cB})$ must be
  affine, which gives the semi-direct product structure as claimed;
   compare \cite[Thm.~5.3]{BBHLN}, of which this is a special case.
  
  That $\cH$ must be non-trivial follows from the observation that the
  unit group $\cO^{\times}\!$ always at least contains the elements
  $\pm 1$. Via the embedding $\iota$, this maps to a non-trivial
  subgroup of $\Aut(\ZZ^2) = \GL(2,\ZZ)$.
\end{proof}

Finally, we can wrap the stabiliser structure as follows. 

\begin{theorem}\label{thm:groups}
  Let\/ $K$ be a quadratic field, with ring of integers\/ $\cO$, and
  let\/ $(\XX^{}_{\cB},\ZZ^2 )$ be the\/ $\cB$-free TDS from
  Proposition~\textnormal{\ref{prop:admissible}}.  Then, the
  normaliser is\/ $\cR = \cS \rtimes \cH$ with\/ $\cS = \cG= \ZZ^2$
  and\/ $\cH \simeq \cO^{\times} \! \rtimes \Aut^{}_{\QQ} (K)$,
  where\/ $\Aut^{}_{\QQ} (K) = \Gal (K/\QQ)$.
\end{theorem}

\begin{proof}
  It is clear from Theorems~\ref{thm:imag} and \ref{thm:real} that
  $\cS \rtimes \cH$ with $\cS = \ZZ^2$ and $\cH$ as stated is a
  subgroup of $\cR$. We need to prove that no other element from
  $\ZZ^2 \rtimes \GL (2,\ZZ)$ can lie in $\cR$.

  Assume, contrary to our claim, that some $(t, M)$ with $t\in\ZZ^2$
  and $M\in \GL (2,\ZZ) \setminus \stab (V'_k)$ lies in $\cR$.  Since
  clearly also $(-t,\one )\in\cR$, we may assume $t=0$.  Now, we will
  generalise the method employed in the proof of
  \cite[Thm.~6.4]{BBHLN} and construct an admissible set $S'\in V'_k$
  such that its image $M(S')$ is not admissible. Then, the unique
  $\ZZ$-linear bijection $A^{}_{M}$ of $\cO$ that corresponds to $M$
  is not an element of $\Aut (\XX_{\cB} )$. From here on, we formulate
  our arguments with $\cO$ and $V_k \subset \cO$ directly, because we
  need to work with ideals anyhow.

  Since $A^{}_M \not\in \stab (V_k)$ by assumption, there is a prime
  ideal $\fp^{}_0$ in $\cO$ and an element $w \in V_k$ such that
  $\fp_0^k$ divides the principal ideal generated by $A^{}_M (w)$.
  Let $P$ be a non-empty finite set of prime ideals of $\cO$ that
  contains all primes $\fp$ with $\No (\fp) < \No (\fp^{}_0)$, but
  none with $\No (\fp) = \No (\fp_0)$. Then, the ideal
  $\fL := \prod_{\fp \in P} \fp^k$ is a submodule of $\cO$ of index
  $\No (\fL) = \prod_{\fp \in P} \No (\fp)^k$.  Since this index is
  coprime to $n := \No (\fp_0^k) = \No (\fp^{}_0)^k$, the ideal $\fL$,
  and likewise its translate $1 + \fL$, meet all cosets of
  $A_M^{-1}(\fp_0^k)$. We set $s^{}_1 := w$ and choose numbers
  $s^{}_2, \dots, s^{}_n \in 1 + \fL$ such that
  $A^{}_M (s^{}_2), \dots, A^{}_M (s^{}_n)$ meet all non-zero cosets
  of $\fp_0^k$. We define $S := \{s^{}_1, \dots s^{}_n \}$.  Then,
  $A^{}_M (S)$ clearly meets all cosets of $\fp_0^k$ and thus is not
  admissible.

  Next, we modify the set $S$ such that it becomes admissible, but
  without changing $A^{}_{M} (S)$ modulo $\fp_0^k$.  Note that
  $S$ is clearly admissible for all primes $\fq$ with
  $\No (\fq) > \No (\fp^{}_0)$ by cardinality. If $S$ happens to meet
  all cosets of $\fp_0^k$, each of them must occur precisely once. We
  then replace $s^{}_2$ by $s'_2:= s^{}_2 + w$ which does not 
  change $A^{}_{M} (S)$ modulo $\fp_0^k$, but reduces the number of
  cosets of $\fp_0^k$ in $S$ by one.

  If there is a second prime $\overline{\fp}^{}_0$ of the same norm,
  $\No (\overline{\fp}^{}_0) = \No (\fp^{}_0)$, and if $S$ happens to
  meet all cosets of $\overline{\fp}_0^k$, we play the same game as
  above. However, due to the previous step, we can neither use $s_2'$
  nor the second element of $S$ which is congruent to $s_2'$ modulo
  $\fp_0^k$. Nevertheless, we still have enough freedom, as $n$ is at
  least $4$.

  It remains to show that $S$ is admissible for all primes $\fp$ with
  $\No (\fp) < \No (\fp_0)$. We know by construction that all
  $s_i \in S$ are congruent to $1$, $w$ or $1+w$ modulo $\fp^k$
  (indeed modulo $\fL$).  It follows that $S$ meets at most $3$ cosets
  of $\fp^k$ and is thus admissible for $\fp$ as
  $\No (\fp^k) \geqslant 2^k \geqslant 4$. Consequently,
  $S' = \iota(S) \subset \ZZ^2$ is the set we were after, and we are
  done.
\end{proof}

Both the centraliser and the normaliser are invariants of topological
dynamical systems, which is to say that topologically conjugate
systems must have isomorphic centralisers and normalisers,
respectively. While the centraliser is always the same in
Theorem~\ref{thm:groups}, hence a toothless tiger in our setting, the
normaliser allows a simple distinction between imaginary quadratic
fields, where $\cO^{\times}$ is a finite group, and real quadratic
fields, where it is not.

{\begin{coro}\label{coro:not-conj}
    Let\/ $k,\ell\geqslant 2$ be arbitrary integers. Then, the\/
    $k$-free shift induced by a real quadratic field can never be
    topologically conjugate to the\/ $\ell$-free shift induced by an
    imaginary quadratic field.  \qed
\end{coro}}

Let us start with a $k$-free shift induced by an arbitrary quadratic
field, $\XX$ say, and consider the hypothetic situation of a factor
shift according to the commutative diagram
\begin{equation}\label{eq:factor}
  \begin{CD}
    \XX  @>\ZZ^{2}>>  \XX \\
    @V\phi VV       @VV\phi V  \\
    \YY  @>\ZZ^{2}>>  \YY
  \end{CD}
\end{equation}
where we assume $\YY$ to be another shift of this kind, hence with the
same translation group acting on $\XX$ and $\YY$. Here, $\phi$ is a
continuous surjection.

While the topological normaliser proves Corollary~\ref{coro:not-conj},
it is a more difficult question whether one such shift can be a factor
of another, according to the diagram in \eqref{eq:factor}. While one
direction can usually be excluded via the topological entropy, as we
shall explain in Section~\ref{sec:entropy} below, also the opposite
one looks highly unlikely. This is so because a factor map in the
presence of such different normalisers would imply an extremely
complicated fibre structure for the mapping $\phi$. We shall discuss
one particular example later.

For a classification of the $k$-free shifts up to topological
conjugacy, we obviously need more than the normaliser. This is clear
from Corollary~\ref{coro:stab} already for a fixed field $K$ and
different values of $k$.  While there are many advanced invariants
around, few of them are easy to determine, and hence of limited
explicit use. One exception is topological entropy, which is a
powerful invariant for the classification task, as we shall discuss
next.

\section{Entropy}\label{sec:entropy}

It is well known \cite{PH} that dynamical systems of this kind have
nice spectral properties, which allow to use the Halmos--von Neumann
theorem for a distinction up to measure-theoretic isomorphism, but not
immediately up to topological conjugacy. Fortunately, one can also
determine the topological entropy. In our number-theoretic setting
with quadratic fields, the result reads as follows.

\begin{theorem}\label{thm:entropy}
  Let\/ $K$ be a quadratic field, and let\/ $V'_{k}$ with\/
  $k\geqslant 2$ be the set defined in
  Eq.~\eqref{eq:def-Vprime}. Then, the topological entropy of the
  induced $\cB$-free TDS\/ $(\XX^{}_{\cB}, \ZZ^2)$ agrees with the
  patch counting entropy of\/ $V'_{k}$ and is given by\/
  $\log(2) \dens(V'_{k})=\log(2)/\zeta^{}_{\nts K} (k)$.
\end{theorem}

\begin{proof}
  It is well known for this type of dynamical systems that the
  topological entropy agrees with the patch-counting entropy; see
  \cite[Thm.~1 and Rem.~2]{BLR} in conjunction with
  \cite[Rem.~4.3]{HR}. As $\XX^{}_{\cB}$ is obtained as an orbit
  closure of the single set $V'_{k}$, we can derive the entropy from
  this set and its properties.
  
  Since $V'_{k}$ is hereditary, we can `knock out' each point
  individually without leaving the space $\XX^{}_{\cB}$, which
  immediately implies that $\log(2) \dens(V'_{k})$ is a lower
  bound for the entropy.

  On the other hand, the set $V'_{k}$ is a weak model set of maximal
  density, with window $W$ in internal space.  In this case, since we
  used a formulation with a lattice of density $1$, we know from
  \cite[Thm.~4.5]{HR} that $\log(2) \vol(W)$ is an upper bound for the
  patch counting entropy of $V^{\prime}_{k}$. This bound also applies
  to the topological entropy of the dynamical system, see
  \cite[Rem.~4.3 and 4.6]{HR}.  Since
  $\vol (W) = 1/\zeta^{}_{\nts K} (k) = \dens(V'_{k})$ by
  Lemma~\ref{lem:density}, our claim follows.
\end{proof}

\begin{lemma}\label{lem:unique}
  Let\/ $K$ be a real quadratic field and let\/ $k = 2\ell\in\NN$ be
  an even integer. Then, $K$ and\/ $k$ are uniquely determined by the
  number\/ $\zeta^{}_{\nts K} (k)$.
\end{lemma}

\begin{proof}
   By a result due to Siegel \cite{Siegel}, see also
   \cite[Ch.~VII, Cor.~9.9]{N},  we know that
 \[
   \zeta^{}_{\nts K} (2 \ell) \, = \, \frac{p \,\pi^{4 \ell}}{q
     \sqrt{d^{}_{\nts K}}}
 \]  
 holds for some coprime $p,q \in \NN$, where $d^{}_{\nts K}$ is the
 discriminant of $K$ as before. Now, if $K$ and $K'$ are both real
 quadratic fields, the identity
 $\zeta^{}_{\nts K} (2 \ell) = \zeta^{}_{\nts K'} (2 \ell{\ts}')$
 implies that $\pi^{4 (\ell{\ts}' - \ell)}$ is algebraic, which forces
 $\ell{\ts}'=\ell$. Then, we get the identity
 $m^2 d^{}_{\nts K} \, = \, n^2 d^{}_{\nts K'}$ for some coprime
 $m,n \in \NN$, hence $m^2 | d^{}_{\nts K'}$ and
 $n^2 | d^{}_{\nts K}$. Consequently, one must have $m,n\in\{1,2\}$,
 and by checking the possible cases one finds that $m=n=1$ is the only
 option.
\end{proof}

This has an interesting consequence on the role of topological
entropy for our dynamical systems as follows.

\begin{prop}
  Among the\/ $k$-free shifts that emerge from real quadratic fields,
  with\/ $k$ even, no two are topologically conjugate unless they are
  equal.  In particular, topological entropy is a complete invariant
  within the class.
\end{prop}

\begin{proof}
	 By Theorem~\ref{thm:entropy} the entropy has the form
	$s = \log (2) / \zeta^{}_{\nts K} (k)$, from which we can derive
	the value $\zeta^{}_{\nts K} (k)$. The latter determines
	$K$ and $k$ by Lemma~\ref{lem:unique}.
\end{proof}

A more general result of this kind involving arbitrary integers
$k \geqslant 2$ might still hold, but deciding this seems to require
new ideas in view of the fact that the known formulas for
$\zeta^{}_K (k)$ at odd positive integers $k$ are too cumbersome;
compare \cite{BCH}. \smallskip

Next, let us show that a folklore conjecture on zeta values predicts
that the entropy is a complete invariant within the class of $k$-free
shifts in imaginary quadratic fields with $k$ odd.  To this end, let
$K$ be an imaginary quadratic field and write $k = 2m + 1$ with
$m\in\NN$.  We want to determine $K$ and $k$ from the value
$\zeta^{}_K(k)$.

Let $\chi \colon \mathrm{Gal}(K/\QQ) \longrightarrow \{\pm 1\}$ be the
unique non-trivial character of $\mathrm{Gal}(K/\QQ)$.  Let $f$ be the
conductor of $\chi$, so $f$ is the smallest positive integer such that
$K$ is contained in the cyclotomic field $\QQ(\xi^{}_f)$, where
$\xi^{}_f$ denotes a primitive $f$-th root of unity. Then, one has
$f = \lvert d^{}_K \rvert$, as follows from \cite[Ch.~VII,
Conductor-Discriminant-Formula 11.9]{N}, but can also be seen more
directly as follows.
	
  We first show that $K$ is contained in $\QQ(\xi^{}_{|d_K|} \, )$.
  Let $K = \QQ(\sqrt{d})$, where $d<0$ is square-free.  If $p$ is an
  odd prime dividing $d_K$, the field
  $\QQ\bigl( \sqrt{(-1)^{(p-1)/2} \, p} \, \bigr)$ is contained in
  $\QQ(\xi^{}_{|d_K|})$; see \cite[p. 51]{Fis}.  Now, consider the
  case $d\equiv 1 \bmod 4$.  Then, $-d$ is a product of distinct
  (positive) odd primes $p^{}_i$, and the number of those $p^{}_i$
  which are $\equiv 3 \bmod 4$ is odd. Hence, we have
\[
  d \, = \, d^{}_K \, =
  \prod_i (-1)^{(p_i-1)/2} \, p^{}_i
\]
and thus $\sqrt{d} \in \QQ(\xi^{}_{|d_K|})$ as desired.  The cases
$d\equiv 2,3 \bmod 4$ can be treated similarly, using the additional
observation that $\QQ \bigl( \sqrt{\pm \ts 2}\, \bigr)$ both are contained
in $\QQ(\xi^{}_8)$.

Conversely, let $f$ be the minimal integer such that $K$ is contained
in $\QQ(\xi^{}_{f})$. By minimality, one has $f \not\equiv 2 \bmod 4$,
as $\QQ(\xi^{}_{f}) = \QQ(\xi^{}_{f/2})$ otherwise.  Moreover, a prime
$p$ is ramified in $\QQ(\xi^{}_{f})$ if and only if $p$ divides
$f$. Since the primes dividing $d_K$ already ramify in $K$, it follows
that $f$ must be divisible by $d_K$, which shows the claim. Only if
$d_K$ is divisible by $8$, this requires a small additional argument
that relies on the fact that none of the fields
$\QQ \bigl( \sqrt{\pm \ts 2}\, \bigr)$ is contained in
$\QQ(\xi^{}_4) = \QQ (\ii)$.

Recall that there is a natural isomorphism
$(\ZZ / \nts f \ZZ)^{\times} \simeq
\mathrm{Gal}\bigl(\QQ(\xi^{}_f)/\QQ\bigr)$ that maps $a \bmod f$ to
the automorphism $\xi^{}_f \mapsto \xi^{a}_f$.  Then, $\chi$ may be
viewed as a Dirichlet character
$(\ZZ /\nts f \ZZ)^{\times} \longrightarrow \CC^{\times}$ with kernel
$\mathrm{Gal}(\QQ(\xi^{}_f) / K)$.  Since $K$ is imaginary, the
character $\chi$ is odd (or has exponent $1$ in the terminology of
\cite[Ch.~VII, \S~2]{N}), as $\chi(-1) = -1 = - \chi(1)$.

Let $\zeta(s) = \zeta^{}_{\QQ}(s)$ and $L(\chi,s)$ be the Riemann zeta
function and the Dirichlet $L$-series attached to $\chi$,
respectively.  Since $k$ is odd and thus congruent to the exponent of
$\chi \bmod 2$, \cite[Ch.~VII, Cor.~10.5 and 2.10]{N} imply that
\begin{equation}\label{eqn:zeta-ratio}
  \frac{\zeta^{}_K (k)}{\zeta (k)} \, = \, L(\chi, k) \, = \,
  (-1)^{m+1}\, \frac{\tau(\chi)}{2 \ii} 
  \left(\frac{2\pi}{f}\right)^k \frac{B^{}_{k,\chi}}{k!} \ts .
\end{equation}
Here, $B^{}_{k,\chi}$ denotes the associated generalised Bernoulli
number which is rational because the image of $\chi$ is,  compare
\cite[p.~441]{N},  and $\tau (\chi)$ is the Gauss sum given by 
\[
  \tau(\chi) \, \defeq \sum_{a \in (\ZZ/\nts f \ZZ)^{\times}} \!
  \chi(a) \, \xi_f^a \, = \sum_{a \in \mathrm{ker}(\chi)} \bigl(
  \xi_f^a - \xi_f^{-a} \bigr) \in \ii \RR \ts .
\]

By \cite[Ch.~VII, Prop.~2.6]{N}, its
absolute value is $\sqrt{f}$, so that indeed
\begin{equation}\label{eqn:Gauss-sum}
  \tau(\chi) \, = \, \pm \ii \sqrt{f} \, =  \, \pm \ii 
  \sqrt{\lvert d^{}_K \rvert} \ts .
\end{equation}
It now follows from \eqref{eqn:zeta-ratio} and \eqref{eqn:Gauss-sum}
that
\begin{equation}\label{eqn:zeta-pi-ratio}
  \alpha^{}_k \, \defeq \, \frac{\zeta^{}_K (k)}{\zeta(k) \pi^k} 
  \, = \, q \sqrt{\lvert d^{}_K \rvert}
\end{equation} 
for some (explicit) non-zero rational number $q$.  It is now
conjectured that the numbers
\[
	\pi, \zeta(3), \zeta(5), \zeta(7), \dots
\]
are algebraically independent (see \cite{Fis} for a survey).  Assuming
this, Eq.~\eqref{eqn:zeta-pi-ratio} shows that the value
$\zeta^{}_K (k)$ determines $k$ uniquely. Once we know $k$, we
retrieve $K$ from
$K = \QQ \bigl( \sqrt{\lvert d^{}_K \rvert} \, \bigr) =
\QQ(\alpha^{}_k)$, where  $\alpha_k$ via
  \eqref{eqn:zeta-pi-ratio} is clearly determined by $\zeta^{}_K (k)$
and $k$. Let us sum this up as follows.

\begin{coro}\label{coro:imag-entropy}
  Let\/ $k,\ell \geqslant 3$ be arbitrary odd integers.  Under the
  assumption that the numbers\/
  $\pi, \zeta(3), \zeta(5), \zeta(7), \ldots $ are algebraically
  independent, no\/ $k$-free shift induced an imaginary quadratic
  field\/ $K$ can be topologically conjugate to the\/ $\ell$-free
  shift induced by\/ $K$ or, in fact, by any other imaginary quadratic
  field.  \qed
\end{coro}

Next, observe that, for $a>1$, the fraction $\frac{1}{1-a^{-s}}$ is
strictly decreasing on the set $\{ s>1\}$. Consequently, the Dedekind
zeta function
\[
  \zeta^{}_{\nts K} (s) \, = \ts \prod_{\fp}
  \myfrac{1}{1-\No (\fp)^{-s}} 
\]
of a quadratic field, which is absolutely convergent on the half-plane
$\{ \mathrm{Re} (s) > 1 \}$, is strictly decreasing on $\{ s >1 \}$ as
well, with $\lim_{s\to \infty} \zeta^{}_{\nts K} (s) = 1$.  This has
the following consequence.

\begin{coro}
  Let\/ $K$ be any quadratic field, and\/ $\XX_k$ the\/ $k$-free shift
  induced by it. Then, the TDS\/ $(\XX_k, \ZZ^2 )$ can never be a
  factor of\/ $(\XX_{\ell}, \ZZ^2 )$ when $k > \ell$.
\end{coro}

\begin{proof}
  The entropy of $\XX_k$ is $\log (2)/\zeta^{}_{\nts K} (k)$ by
  Theorem~\ref{thm:entropy}. As such, via the above observation, it is
  strictly increasing on $\{ k\in\NN : k \geqslant 2 \}$, with
  limiting value $\log (2)$ as $k\to\infty$.

  Since no factor of a TDS, in the sense of \eqref{eq:factor}, can
  have a larger entropy than the original TDS, compare
  \cite[Prop.~10.1.3]{VO}, the claim is immediate.
\end{proof}

Let us close with another example, where we consider the shift
$\XX_V$ induced by the visible lattice points 
$V  =  \{ (m,n) \in \ZZ^2 : (m,n) = 1 \}
    =  \ZZ^2 \setminus \bigcup_{p} p\ts \ZZ^2$,
where $p$ runs through the rational primes,
in comparison to the shift $\XX_{\mathrm{G}}$
induced by the square-free Gaussian integers.

\begin{prop}
  Neither of the two shifts\/ $(\XX_V,\ZZ^2)$ and\/
  $(\XX_{\mathrm{G}},\ZZ^2)$ can be a topological factor of the other
  in the sense of the diagram in \eqref{eq:factor}.
\end{prop}

\begin{proof}
  The entropy of $(\XX_V,\ZZ^2)$ is $\log (2) / \zeta (2)$; compare
  \cite{BH}. Now, $\zeta^{}_{\QQ (\ii)} (2) = \zeta (2) \, L(\chi, 2)$
  with $\chi$ the principal character of $\QQ (\ii)$, where
  $L (\chi, 2) < 1$. Consequently, as the topological entropy of the
  Gaussian shift is larger than that of $\XX_V$, one direction is
  ruled out immediately.
 
  For the other direction, assuming that we have a surjective factor
  map $\phi \colon \, \XX_{\mathrm{G}} \longrightarrow \XX_V$, we now
  construct a configuration that is legal in $\XX_{\mathrm{G}}$ whose
  image under $\phi$ cannot lie in $\XX_V$. Due to the CHL theorem,
  $\phi$ is a sliding block map, hence given by a local function
  $\varPhi\colon \, \{0,1\}^M \longrightarrow \{0,1\}$, where
  $M\nts\subset \ZZ^2$ is the \emph{memory set} or \emph{local window}
  of $\phi$.

  Since $\phi$ is surjective, the singleton pattern
  ${1}^{}_{\{\bs{0}\}} \in \XX_V$ must have some preimage in
  $\XX_{\mathrm{G}}$. This, in turn, implies the existence of some
  point pattern $P$ such that $\varPhi (P) = 1$. As $P$ is a pattern
  that appears in some element of $\XX_{\mathrm{G}}$, we may identify
  $P$ with a G-admissible finite subset of $M$, which we also call $P$
  by slight abuse of notation. Thus, for every Gaussian prime $q$,
  there is some coset $r^{}_{\nts q} + (q^2)$ whose intersection with
  $P$ is empty.

  In what follows, let $p$ be an inert rational prime such that
  $p^2 > \card ( P) $.  Let $\{ q^{}_1, q^{}_2, \ldots, q^{}_k \}$ be
  the set of all Gaussian primes of norm less than $p^2$, which is a
  finite set.  By the Chinese Remainder Theorem, for every element
  $(m, n)$ with $0 \leqslant m,n \leqslant p-1$, there exists a unique
  solution
  $\bmod \, {\bigl( q_1^2 \ts q_2^2 \cdots q_k^2 \ts p^2 \bigr)}$ in
  $\ZZ [\ii]\simeq \ZZ^2$ for the system of equations given by
\begin{align*}
  x \equiv & \; 0 \;  \bmod \bigl( q^2_i \bigr) \quad
      \text{for all } 1 \leqslant i \leqslant k \quad \text{and} \\
  x \equiv & \;  (m,n)  \;  \bmod \bigl( p^2 \bigr).
\end{align*}
Let $x^{(m,n)}\in\ZZ^2$ be one solution of this system of
congruences. Then, the set of \emph{all} solutions is the lattice
coset $x^{(m,n)}+ \bigl( q_1^2 \ts q_2^2\cdots q_k^2 \ts p^2
\bigr)$. Clearly, no translation by an element of the lattice
$\bigl( q_1^2 \ts q_2^2\cdots q_k^2 \ts p^2 \bigr)$ changes the
equivalence class $\bmod{\bigl(q_i^2 \bigr)} $ of any element of $P$.
Now, for any $q \in \{ q^{}_1 , \ldots , q^{}_k \}$, there is some
coset $r^{}_{\nts q} + \bigl( q^2 \bigr)$ that has empty intersection
with $P$, whence we also have the relation
$(x^{(m,n)} + P)\cap \bigl( r^{}_q + (q^2) \bigr)=\varnothing$.

Clearly,
$x^{(m,n)}+ \bigl( q_1^2 \ts q_2^2 \cdots q_k^2 \ts p^2 \bigr)$ is a
relatively dense subset of $\ZZ [\ii]$. Then, for every $(m, n)$ with
$0\leqslant m,n \leqslant p-1$, we can choose an element
$y^{(m, n)} \in x^{(m, n)} + \bigl( q_1^2 \ts q_2^2 \cdots q_k^2 \ts
p^2 \bigr)$ such that
$\| y^{(m, n)} - y^{(m', n')} \|^{}_{2} > 2 \diam (M)$ holds for
$(m, n) \neq (m', n')$.  Let us now consider the set
\[ 
   P^{*} \, = \bigcup_{ 0 \leqslant m, n \leqslant p - 1}
   y^{(m, n)} + P ,
\]
where we recall that we identify $P$ with a subset of $M$.
Since every term of this disjoint union has empty intersection with
$r^{}_q + \bigl( q^2 \bigr)$, where $q$ is any of the $q^{}_{i}$, the
set $P^{*}$ has empty intersection with this coset as well.

Furthermore, $\card (P^{*}) = p^2 \ts \card (P) < p^4$, due to our
choice of $p$. Since the prime $p$ is inert, and thus also a Gaussian
prime, we have $\bigl[ \ZZ[\ii] : (p^2) \bigr] = \No (p)^2 = p^4$, so
$P^{*}$ necessarily misses a coset of $(p^2)$.  The same holds for
every Gaussian prime of norm larger than $p^2$. Since any Gaussian
prime of norm smaller than $p^2$ is one of the $q^{}_i$'s, we conclude
that $P^{*}$ is G-admissible, whence
$u = 1^{}_{\nts P^{*}} \in \XX^{}_{\mathrm{G}}$.  By the
choice of $\varPhi$ and the CHL theorem,
we then have
\[ 
   \phi (u)_{y^{(m, n)}} \, = \,
   \varPhi \bigl( u |_{y^{(m, n)} + M} \bigr)
   \, = \, \varPhi (P) \, = \, 1 \ts ,
\]
for every $0\leqslant m,n \leqslant p-1$, where we use that the $p^2$
translates of $P$ in $P^{*}$ are separated by more than $2 \diam (M)$
by construction. Thus, they locally (for a disk-like window that
covers the set $M$) look like a translate of $1^{}_{\nts
  P}$. Consequently, if $\phi (u) = 1^{}_U$, we have the inclusion
$U \supseteq \{ y^{(m, n)} : 0 \leqslant m,n \leqslant p-1\}$.

Since we have  $y^{(m, n)} \equiv (m,n) \bmod (p^2)$,
with $(p^2) = p^2 \ZZ^2$, we
also have $y^{(m,n)} \equiv (m, n) \bmod p\ZZ^2$.  Thus, $U$ contains
a complete set of representatives of $\ZZ^2 / p\ts \ZZ^2$, hence
cannot be $V$-admissible. This contradiction implies that the factor
map $\phi$ cannot exist.
\end{proof}

It is clear that this argument can be adapted to other quadratic
fields as well, which we leave to the interested reader. It seems
quite plausible that most if not all of the shift spaces we have
analysed above are independent of each other in this stronger sense.

At this point, it is also natural to cover more field extensions that
are Galois, and consider general cyclotomic fields in
particular. Here, we expect that the result on the (extended)
symmetries is structurally the same, which suggests that it might hold
more generally. On the other hand, it looks doubtful whether entropy
can be as strong as it seems here.

\section*{Acknowledgements}

It is our pleasure to thank Uwe Grimm, J\"urgen Kl\"uners, Christoph
Richard, Dan Rust and Johannes Sprang for helpful discussions.  We
  thank an anonymous referee for several thoughtful comments that
  helped us to improve the presentation. AB  was supported by
  the EPSRC (grant no.~EP/S010335/1). He is also grateful to the
Research Centre for Mathematical Modelling (RCM${}^2$) at Bielefeld
University for hospitality, and AN acknowledges financial support by
the German Research Foundation (DFG) through its Heisenberg program
(project no.~437113953).

\end{document}